\documentclass[11pt]{amsart}
\usepackage{etex}
\usepackage{diagbox}
\usepackage{array}
\usepackage{graphicx, pictex}
\usepackage{stmaryrd}
\usepackage{changepage}
\usepackage{float}
\restylefloat{table}
\usepackage{multirow}
\usepackage[dvipsnames]{xcolor}
\usepackage{amsmath,amssymb,amsthm,mathrsfs}
\allowdisplaybreaks

\renewcommand{\d}{\partial}

\def\d{\Omega}

\def\du#1#2#3{\overset{#3}{\underset{#2}{#1}}}
\def\Forall{\quad \hbox{ for all }}

\def\M{{\mathcal{M}}}
\def\B{{\mathcal{B}}}

\newcommand{\tn}[1]{\lVert\kern-1pt\lvert{#1}\rvert\kern-1pt\rVert}
\def\<{{\langle}}
\def\>{{\rangle}}
\def\Forall{\quad \hbox{ for all }}

\def\d{\Omega}

\def\Forall{\quad \hbox{ for all }}

\def\d{\Omega}

\def\Forall{\quad \hbox{ for all }}

\def\tb#1{{\|\kern-1pt| #1 \|\kern-1pt|}}
\def\nm2#1#2{\|#1\|_{2,\d_{#2}}}

\def\R{\mathbb{R}}

 \theoremstyle{plain}
 \newtheorem{thm}{Theorem}[section]
 \numberwithin{equation}{section} 
 \numberwithin{figure}{section} 
 \theoremstyle{plain}
 \newtheorem{prop}[thm]{Proposition}
 \theoremstyle{plain}
 \theoremstyle{plain}
 \newtheorem{theorem}[thm]{Theorem}
 \theoremstyle{plain}
 
\theoremstyle{plain}
 
 \theoremstyle{plain}
 
\def\M{{\mathcal{M}}}

\def\d{{\Omega}}

\def\Forall{\quad \hbox{ for all }}

\def\<{{\langle}}
\def\>{{\rangle}}
\def\R{\mathbb{R}}

\def\du#1#2#3{\overset{#3}{\underset{#2}{#1}}}

\usepackage{amssymb}

\begin{document}

\title[Analysis for Mixed FE Discretization for Convection-Diffusion]
{Results on a Mixed Finite Element Approach for a Model Convection-Diffusion Problem}
\author{Constantin Bacuta}
\address{University of Delaware,
Mathematical Sciences,
501 Ewing Hall, 19716}
\email{bacuta@udel.edu}

\author{Daniel Hayes}
\address{University of Delaware,
Department of Mathematics,
501 Ewing Hall, 19716}
\email{dphayes@udel.edu}

\author{Tyler O'Grady}
\address{University of Delaware,
Department of Mathematics,
501 Ewing Hall, 19716}
\email{togrady@udel.edu}

\keywords{least squares, saddle point systems, mixed methods,  optimal stability norm, convection dominated problem}

\subjclass[2000]{74S05, 74B05, 65N22, 65N55}
\thanks{The work was supported  by NSF-DMS 2011615}%

\begin{abstract}

We consider  a model convection-diffusion problem  and \\ present our recent numerical and analysis results regarding mixed finite element formulation and discretization in the singular perturbed case when the convection term dominates the problem. Using  the concepts of optimal norm and saddle point reformulation, we found new error  estimates  for the case of uniform meshes.  We compare  the standard linear Galerkin discretization to a saddle point least square discretization that uses quadratic test functions, and explain the non-physical oscillations of the discrete solutions. We also relate a known upwinding Petrov-Galerkin method and  the stream-line diffusion discretization method, by emphasizing the resulting  linear  systems and by comparing appropriate error norms.  The results can be extended to the multidimensional case in order to find efficient approximations for more general singular perturbed problems including convection dominated models. 

\end{abstract}
\maketitle

\section{Introduction}


We start with the model of a  singularly perturbed  convection diffusion problem: Find $u=u(x)$  on $[0, 1]$ such that
\begin{equation}\label{eq:1d-model}
\begin{cases}-\varepsilon u''(x)+u'(x)=f(x),& 0<x<1\\ u(0)=0, \ u(1)=0,\end{cases}
\end{equation} 
in  the convection dominated case, i.e.  $\varepsilon \ll 1$. Here, the function $f$ is given and assumed to be square integrable on $[0, 1]$. We  will use the following  notation:
\[ 
\begin{aligned}
a_0(u, v) & = \int_0^1 u'(x) v'(x) \, dx, \ (f, v) = \int_0^1  f(x) v(x) \, dx,\ \text{and}\\
b(v, u)& =\varepsilon\, a_0(u, v)+(u',v)  \ \text{for all} \ u,v \in V:=H^{1}_0(0,1).
\end{aligned}
 \]
A variational formulation of \eqref{eq:1d-model} is : Find $u \in V:= H_0^1(0,1)$ such that
 \begin{equation}\label{VF1d}
b(v,u) = (f, v), \ \text{for all} \ v \in V.
\end{equation}
The discretization of \eqref{VF1d}, and its multi-dimensional  variants arise  when solving  practical PDE models such as  heat transfer problems in thin domains, as well as when using small step sizes in implicit time discretizations of parabolic convection diffusion type problems, \cite{Lin-Stynes12}. The solutions to these problems are characterized by  boundary layers,  see e.g., 
\cite{Dahmen-Welper-Cohen12, linssT10, Roos-Schopf15}. 
Approximating such solutions poses  numerical  challenges due to the $\varepsilon$-dependence of both the  error estimates  and of the  stability constants. 
The goal of the paper is to investigate  finite element discretization of a model convection diffusion problem that proved  to be a challenging problem for the last  few  decades, see e.g., \cite{ Barrett-Morton81,  Dahmen-Welper-Cohen12, Roos-Schopf15, Chan-Heuer-Bui-Demkowicz14}. The focus is  on  analysis  of the variational problem that is written in a  mixed formulation  and  ledas to new stability and approximation results.  To improve the rate of convergence in particular norms, we will use the concept of optimal norm, see e.g., \cite{Broersen-StevensonDPGcd14, Chan-Heuer-Bui-Demkowicz14, Dahmen-Welper-Cohen12, demkowicz-gopalakrishnanDPG10, dem-fuh-heu-tia19,J5onDPG, Barrett-Morton81, BHJ, BHJ22}, that provides 
 $\varepsilon$-independent stability. In addition,  we will take advantage of the  mixed  reformulations of the  variational problem given by the Saddle Point Least Squares (SPLS) method, as presented in  \cite{BJ-nc, BJprec, BJ-AA, BQ15}. The ideas, concepts, and methods we present here, can be extended  to the multidimensional case, leading to new and efficient finite element discretizations for convection dominated problems. 


The SPLS approach  uses an auxiliary variable that represents the residual of the original variational formulation on the test space and adds another simple equation involving the residual variable. The method  leads to a square symmetric saddle point system that is more suitable for analysis and discretization. 
The SPLS method was used succesfully for more general boundary value problems  problems, see e.g., \cite{BM12,Dahmen-Welper-Cohen12,demkowicz, Barrett-Morton81}.  Many of the aspects  regarding  SPLS formulation  are common to both the DPG approach \cite{ bouma-gopalakrisnan-harb2014,car-dem-gop16,demkowicz-gopalakrishnanDPG10, demkowicz2011class, dem-fuh-heu-tia19,J5onDPG}  and the SPLS approach developed in  \cite{BJ-nc, BJprec, BJ-AA, BQ15}. In our  work here, the concept of optimal norms will play a key role in providing a  unified error analysis for a class of  finite element discretizations of convection-diffusion problems. 
 
The paper is organized as follows.  We review the main ideas  of the SPLS approach in an abstract general setting in Section \ref{sec:ReviewLSPP}. In Section \ref{sec:abstract-discretization}, we present the SPLS discretization together with some general  error approximation results. We also prove a new approximation result for the Petrov-Galerkin case when the norms on the continuous and discrete test spaces are different. Section \ref{sec:1d-lin-discrete} reviews and connects four known discretization methods that have  $C^0-P^1$ as a trial space, and are to be analyzed  as mixed methods.  Using  various numerical test, we illustrate and explain the non-physical oscillation phenomena  for the standard and SPLS discretization. In addition, we show  the strong connection between the upwinding Petrov-Galerkin (PG) and  the stream-line diffusion (SD) methods.  Section \ref{sec:stability}, focuses on the study of  the stability and approximability of the mixed  discretizations. Numerical results are  presented in Section \ref{sec:NR}.

\section{The notation and the general SPLS approach}\label{sec:ReviewLSPP}
In this section we present the main ideas and concepts for the  SPLS method for a general  mixed variational formulation.
We follow the  Saddle Point Least Squares (SPLS)  terminology that was introduced in  \cite{BJ-nc, BJprec, BJ-AA, BQ15}. 
\subsection{The abstract variational formulation at the continuous level}  \label{subsec:mh}
We consider an abstract  mixed or Petrov-Galerkin formulation that  generalizes the formulation \eqref{VF1d}:
Find $u \in Q$ such that
 \begin{equation}\label{VFabstract}
b(v,u) = \<F, v\>, \ \text{for all} \ v \in V.
\end{equation}
where $b(\cdot,\cdot)$ is a bilinear form, $Q$ and $V$ are possible different  separable Hilbert spaces and $F$ is a continuous linear functional on $V$.
We assume that the inner products $a_0(\cdot, \cdot)$ and $(\cdot, \cdot)_{{Q}}$ induce the  norms $|\cdot|_V =|\cdot| =a_0(\cdot, \cdot)^{1/2}$ and $\|\cdot\|_{Q}=\|\cdot\|=(\cdot, \cdot)_{Q}^{1/2}$. We denote the dual of $V$ by $V^*$ and the dual pairing on $V^* \times V$ by $\langle \cdot, \cdot \rangle$.  
We assume that $b(\cdot, \cdot)$ is a continuous bilinear form on $V\times Q$ satisfying
 the $\sup-\sup$ condition
 \begin{equation} \label{sup-sup_a}
\du{\sup}{u \in Q}{} \ \du {\sup} {v \in V}{} \ \frac {b(v, u)}{|v|\,\|u\|} =M <\infty, 
\end{equation} 
and the $\inf-\sup$ condition
 \begin{equation} \label{inf-sup_a}
 \du{\inf}{u \in Q}{} \ \du {\sup} {v \in V}{} \ \frac {b(v, u)}{|v|\,\|u\|} =m>0.
\end{equation}

With the form $b$, we associate the operators $\B:V\to {Q}$    defined by  
\[
(\B v,q)_{{Q}}=b(v, q) \,  \Forall  v \in V, q \in Q.
\]
We define $V_0$ to be the kernel of $\B$, i.e.,
\[ 
V_0 :=Ker(\B)= \{v \in V |\  \B v=0\}.
\]
Under assumptions \eqref{sup-sup_a}  and \eqref{inf-sup_a}, the operator $\B$ is a bounded surjective operator from $V$ to $Q$, and $V_0$ is a closed subspace of $V$. We will also assume that the  data $F \in V^*$ satisfies the {\it compatibility condition} 
\begin{equation}\label{eq:BBsuf}
\<F,v\> =0 \Forall v \in V_0=Ker(\B). 
\end{equation}
  
The following result describes the well posedness  of  \eqref{VFabstract} and can be used at the continuous  and discrete levels, see e.g.  \cite{A-B, B09, boffi-brezzi-demkowicz-duran-falk-fortin2006, boffi-brezzi-fortin}. 
\begin{prop} \label{prop:well4mixed} If the form $b(\cdot,\cdot)$ satisfies \eqref{sup-sup_a} and \eqref{inf-sup_a}, and the  data $F \in V^*$ satisfies the {\it compatibility condition}  \eqref{eq:BBsuf}, then  the problem \eqref{VFabstract} has  unique solution that depends continuously on the data $F$. 
\end{prop}
It is also known, see e.g., \cite{BM12, BQ15,BQ17, Dahmen-Welper-Cohen12} that, under the  {\it compatibility condition} \eqref{eq:BBsuf}, solving the mixed problem  \eqref{VFabstract} reduces to  solving a standard saddle point reformulation: Find $(w, u) \in V \times Q$ such that  
\begin{equation}\label{abstract:variational2}
\begin{array}{lclll}
a_0(w,v) & + & b( v, u) &= \langle F,v \rangle &\ \Forall  v \in V,\\
b(w,q) & & & =0   &\  \Forall  q \in Q.  
\end{array}
\end{equation}
In fact,  we have that $u$ is the unique solution of \eqref{VFabstract}  {\it if and only if} $(w=0 , u) $ solves  \eqref{abstract:variational2}, and the result remains valid if the form $a_0(\cdot, \cdot)$ in \eqref{abstract:variational2} is replaced by any other symmetric bilinear form  $a(\cdot, \cdot)$  on $V$ that leads to an equivalent norm on $V$. 




\section{Saddle point least squares discretization}\label{sec:abstract-discretization}

Let $b(\cdot,\cdot):V\times Q\to\R$ be a bilinear form as defined in Section  \ref{sec:ReviewLSPP}. 
Let $V_h\subset V$ and  $\M_h\subset Q$ be finite dimensional approximation spaces.
 We assume the following discrete $\inf-\sup$ condition holds for the pair of spaces $(V_h,\M_h)$:
 \begin{equation} \label{inf-sup_h}
\du{\inf}{u_h \in \M_h}{} \ \du {\sup} {v_h \in V_h}{} \ \frac {b(v_h, u_h)}{|v_h|\,\|u_h\|} =m_h>0.
\end{equation} 
As in the continuous case, we  define 
\[
V_{h,0}:=\{v_h\in V_h\,|\, b(v_h,q_h)=0,\Forall q_h\in \M_h\},
\] 
 and $F_h \in V_h^*$ to be the restriction of $F$ to $V_h$, i.e.,   $\langle F_h, v_h \rangle:=\langle F, v_h \rangle$ for all $v_h \in V_h$. 
In the case $V_{h,0} \subset V_0$, the compatibility condition \eqref{eq:BBsuf} implies the discrete compatibility condition
\[
\langle F,v_h\rangle =0 \Forall v_h\in V_{h,0}.
\]
Hence, under assumption \eqref{inf-sup_h}, the PG  problem of finding $u_h\in \M_h$ such that 
\begin{equation}\label{discrete_var_form}
b(v_h, u_h)=\langle F,v_h\rangle,  \ v_h\in V_h
\end{equation}
has a unique solution. In general, we might not have  $V_{h,0}\subset V_0$. Consequently,  even though the continuous problem \eqref{VFabstract} is well posed, the discrete problem  \eqref{discrete_var_form} might not be  well-posed. However, if the form $b(\cdot, \cdot)$ satisfies \eqref{inf-sup_h}, then the problem of finding $(w_h,u_h) \in V_h\times \M_h$ satisfying  
\begin{equation}\label{discrete:variationalSPP}
\begin{array}{lclll}
a_0(w_h,v_h) & + & b( v_h, u_h) &= \langle f,v_h \rangle &\ \Forall  v_h \in V_h,\\
b(w_h,q_h) & & & =0   &\  \Forall  q_h \in \M_h, 
\end{array}
\end{equation} 
 does have a unique solution. 
We call  the component $u_h$  of the solution $(w_h,u_h)$ 
of \eqref{discrete:variationalSPP} the  {\it saddle point least squares} approximation of the solution $u$ of the original mixed problem \eqref{VFabstract}.

The following error estimate for $\|u-u_h\|$ was proved in \cite{BQ15}. 

 
\begin{theorem}\label{th:sharpEE} 
Let $b:V \times Q \to \R$  satisfy \eqref{sup-sup_a} and \eqref{inf-sup_a} and assume that   ${F}  \in V^*$  is given and satisfies \eqref{eq:BBsuf}. Assume that  $u$  is the  solution  of \eqref{VFabstract} and  $V_h \subset V$,  $ {\M}_h \subset  Q$ are  chosen such that the discrete $\inf-\sup$ condition   \eqref{inf-sup_h} holds. If  $\left (w_h, u_h \right )$ is the  solution  of \eqref{discrete:variationalSPP}, then the following error estimate holds:
\begin{equation}\label{eq:er4LS} 
 \frac 1 M |w_h| \leq \|u-u_h\| \leq  \frac{M}{m_h} \  \du{\inf}{q_h \in\M_h
}{}  \|u-q_h\|.
\end{equation} 
\end{theorem}
 
Note that the considerations made so far in this section remain valid  if the form $a_0(\cdot, \cdot)$, as an inner product on $V_h$, is replaced by another inner product $a(\cdot, \cdot)$ which gives rise to an equivalent  norm on $V_h$. 

For the case  $V_{h,0} =\{0\}$, the compatibility condition \eqref{eq:BBsuf} is trivially satisfied and there is no need for an SPLS discretization, unless we want to precondition the  discretization \eqref{discrete_var_form}. Thus, \eqref{discrete_var_form} leads to a square linear system that is the  Petrov-Galerkin discreization of \eqref{VFabstract}. In this case,  we might have a different  norm 
$ \|\cdot \|_*  $ on $Q$  and a different norm $ \|\cdot \|_{*,h}  $ on the discrete trial space $\M_h$. The approximability  Theorem \ref{th:sharpEE} can be adapted in this case to  the following version: 

\begin{theorem}\label{th:ap-PG}
Let $|\cdot|$, $\|\cdot\|=\|\cdot\|_{*}$ and $\|\cdot\|_{*,h}$ be the norms on $V,Q$, and $\M_h$ respectively such that  they satisfy \eqref{sup-sup_a}, \eqref{inf-sup_a}, and \eqref{inf-sup_h}. Assume that for some constant $c_0$ we have
\begin{equation}
\|v\|_* \leq c_0\|v\|_{*,h}\quad\quad\text{for all $v\in Q$}.
\end{equation}
Let $u$ be the solution of \eqref{VFabstract} and let  $u_h$ be  the unique solution of problem \eqref{discrete_var_form}. Then the following error estimate holds:
\begin{equation}\label{eq:Approx-Disc}
\|u-u_h\|_{*,h}\leq c_0\, \frac{M}{m_h}\ \du{\inf}{p_h \in \M_h}{} \ \|u-p_h\|_{*,h}.
\end{equation}
\end{theorem}
\begin{proof}
Let $T_h:Q\to Q$ be the operator defined by $T_hu =u_h$ where $b(v_h,u) = b(v_h,u_h)$ for all $v_h\in V_h$. On $Q$ we consider the norm $\|\cdot\|_{*,h}$. By  the uniqueness of the discrete solution to the problem ``Find $\tilde{u}_h\in \M_h$ such that
$$
b(v_h,\tilde{u}_h) = b(v_h,u_h),\quad\quad\text{for all $v_h\in V_h$},
$$
we have that $T_hu_h= u_h$, i.e. $T_h^2 = T_h$. Using that   $\|I-T_h\|_{\mathcal{L}}= \|T_h\|_{\mathcal{L}}$
, where $\|\cdot\|_{\mathcal{L}}=\|\cdot\|_{\mathcal{L}(Q,Q)}$, see \cite{kato, xu-zikatanov-BBtheory}, we get
\begin{align*}
\|u-u_h\|_{*,h} & = \|(I-T_h)u\|_{*,h}
 = \|(I-T_h)(u-q_h)\|_{*,h} \\
& \leq \|I-T_h\|_{\mathcal{L}}\, \|u - p_h\|_{*,h}
= \|T_h\|_{\mathcal{L}}\, \|u-p_h\|_{*,h}
\end{align*}
where $p_h$ is any element of $\M_h$. Thus, we need a bound for $\|T_h\|_{\mathcal{L}}$:
\begin{align*}
\|T_hu\|_{*,h} & \leq \frac{1}{m_h}\inf_{v_h\in V_h}\frac{b(v_h,u_h)}{|v_h|} 
 = \frac{1}{m_h}\inf_{v_h\in V_h}\frac{b(v_h,u)}{|v_h|}\\
& \leq \frac{M}{m_h}\|u\|_{*}  \leq \frac{c_0M}{m_h}\|u\|_{*,h}.
\end{align*}
By combining the last two estimates,  we have:
\begin{equation}\label{eq:approx-proof}
\|u-u_h\|_{*,h}\leq c_0\, \frac{M}{m_h}\|u - p_h\|_{*,h}
\end{equation}
Since $p_h\in \M_h$ was arbitrary, we obtain \eqref{eq:Approx-Disc}.
\end{proof}


\section{Discretization with $C^0-P^1$ trial space for the 1D Convection reaction problem}\label{sec:1d-lin-discrete}
In this section, we review standard finite element discretizations of problem \eqref{eq:1d-model} and emphasize the ways the corresponding linear systems relate. The concepts presented in this section are focused on uniform mesh discretization, but most of the results can be easily  extended to non-uniform meshes. 

We  divide the interval $[0,1]$ into $n$  equal length subintervals using the nodes $0=x_0<x_1<\cdots < x_n=1$ and denote  $h:=x_j - x_{j-1}, j=1, 2, \cdots, n$. For the above uniform distributed notes on $[0, 1]$, we define  the corresponding discrete space  $\M_h$  as  the subspace of $Q = H^1_0(0,1)$, given by
 \[ 
 \M_h = \{ v_h \in V \mid v_h \text{ is linear on each } [x_j, x_{j + 1}]\},
 \]
  i.e., $\M_h$ is the space of all {\it piecewise linear continuous functions} with respect to the given nodes, that are zero at $x=0$ and $x=1$.  We consider the nodal basis $\{ \varphi_j\}_{j = 1}^{n-1}$ with the standard defining property $\varphi_i(x_j ) = \delta_{ij}$. 

\subsection{Standard Linear discretization} \label{sec:LinP1}
We couple the above discrete trial space with a discrete  test space $V_h:=\M_h$.  
 Thus, the standard linear discrete variational formulation of \eqref{VF1d} is: Find $u_h \in \M_h$ such that
 \begin{equation}\label{dVF}
b(v_h, u_h) = (f, v_h), \ \text{for all} \ v_h \in V_h.
\end{equation}
We look for  $u_h \in V_h$ with  the nodal basis expansion
 \[
 u_h := \sum_{i=1}^{n-1} u_i \varphi_i, \ \text{where} \ \ u_i=u_h(x_i).
 \]
If we consider the test functions $v_h=\varphi_j, j=1,2,\cdots,n-1$ in \eqref{dVF}, we  obtain  the following linear system 
 \begin{equation}\label{1d-LS}
 \left (\frac{\varepsilon}{h}  S+ C \right )\, U = F, 
\end{equation}

where \(U,F\in\R^{n-1}\) and \(S, C \in\R^{(n-1)\times(n-1)}\) with:
 \[
 U:=\begin{bmatrix}u_1\\u_2\\\vdots\\u_{n-1}\end{bmatrix},\quad F:=\begin{bmatrix}(f,\varphi_1)\\ (f,\varphi_2)\\\vdots \\ (f,\varphi_{n-1})\end{bmatrix}, \text{and} 
\]
 \[
 S:=\begin{bmatrix}2&-1\\-1&2&-1\\&\ddots&\ddots&\ddots\\&&-1&2&-1\\&&&-1&2\end{bmatrix},\quad 
 C:=\frac{1}{2}\begin{bmatrix}0&1\\-1&0&1\\&\ddots&\ddots&\ddots\\&&-1&0&1\\&&&-1&0\end{bmatrix}.\] 
Note that, by letting   $\varepsilon \to 0$  in \eqref{VF1d}, we obtain the {\it simplified problem}: \\
Find $w \in H_0^1(0,1)$ such that
\begin{equation} \label{VF1d-simplified-w}
(w',v) = (f, v), \ \text{for all} \ v \in V. 
\end{equation}

The problem \eqref{VF1d-simplified-w} has unique solution,  if and only if $\int_0^1  f(x)  \, dx=0$.    For the case $\int_0^1  f(x)  \, dx \neq 0$ we can consider  the  {\it reduced problem}: \\ Find $w \in H^1(0,1)$ such that
 \begin{equation}\label{VF1d-reduced}
w'(x) = f(x)  \ \text{for all} \ x \in (0, 1), \text{and} \ w(0)=0,
\end{equation}
with the unique solution $w(x) = \int_0^x  f(s)  \, ds$. 

The  {\it simplified discrete problem} corresponding to the  finite element discretization \eqref{VF1d-simplified-w} can be written as:  Find $w_h := \sum_{i=1}^{n-1} u_i \varphi_i$, such that 
\begin{equation}\label{lin-reduced}
C\, U = F.
\end{equation}
It is interesting to note that, even though, in general, \eqref{VF1d-simplified-w} is not well posed, the  system \eqref{lin-reduced} decouples into two independent systems:
\begin{equation}\label{sys-even}
\begin{cases} u_2 -u_0 &=2(f, \varphi_1) \\  u_4 -u_2&=2(f, \varphi_3)\\ \vdots \\ 
u_{2m} -u_{2m-2}&=2(f, \varphi_{2m-1}), 
\end{cases}
\end{equation}
and
 \begin{equation}\label{sys-odd}
\begin{cases} u_3 -u_1 &=2(f, \varphi_2) \\  u_5 -u_3&=2(f, \varphi_4)\\ \vdots \\ 
u_{2m+1} -u_{2m-1}&=2(f, \varphi_{2m}),
\end{cases}
\end{equation}
where $u_0=u_n=0$.  In this case, the systems \eqref{sys-even} and \eqref{sys-odd} have  unique solutions and can be  solved, forward and backward respectively, to get 

 \begin{equation}\label{sys-sol-odd}
\begin{cases} u_{2k} &=2 \sum_{j=1}^k (f, \varphi_{2j-1}), \ k=1,2,\cdots,m \\ 
 u_{2m-2k+1} &=-2 \sum_{j=1}^{k} (f, \varphi_{2m-2j+2}), \ k=1,2, \cdots,m
\end{cases}
\end{equation}
For $f=1$ on $[0, 1]$, we have $(f,\varphi_i) =h$ for all $i=1,2, \cdots, 2m$, and 
 \begin{equation}\label{sys-sol-odd-f1}
\begin{cases} u_{2k} &=2kh=x_{2k}, \ k=1,2,\cdots,m \\ 
 u_{2m-2k+1} &=- 2kh =x_{2m-2k+1}-1, \ k=1,2, \cdots,m.
\end{cases}
\end{equation}
Thus, the even components interpolate the  solution of the function $x$, and the odd components interpolate the function $x-1$.
The combined solution leads to a very oscillatory behavior when $n\to \infty$.  For $\varepsilon/h \leq 10^{-4}$, the solution of \eqref{dVF} is very close to the solution of the simplified  system \eqref{lin-reduced}.  A similar oscillatory behavior  is observed for the linear  finite element solution of \eqref{dVF} when using an odd number of subintervals $n$, see Fig.1. 

We note  that, for an arbitrary smooth $f$,  the even components $\{u_{2k}\}$ approximate  the solution $w(x)$ of the Initial Value Problem (IVP)  \eqref{VF1d-reduced}, and the odd components approximate  the function $\theta(x)=w(x)- \int_0^1 f(x)\, dx$, see Fig.1 and Fig.5, where
\begin{equation}\label{VF1d-reduced2}
\theta'(x) = f(x)  \ \text{for all} \ x \in (0, 1), \text{and} \ \theta(1)=0.
\end{equation}
This can be justified  as follows. If we replace in \eqref{sys-even} the values $(f,\varphi_i)$  by   $h\, f(x_i)$ - the corresponding trapezoid rule approximation of the integral, the solution of the modified system coincides with the  mid-point method approximation 
 of the IVP \eqref{VF1d-reduced}, (on the even nodes, $h\to 2h$). Similarly,  the  solution of the modified  system \eqref{sys-odd} obtained by replacing $(f,\varphi_i)$  with  $h\, f(x_i)$ coincides with the mid-point method approximation of  the IVP on odd nodes. 
\vspace{0.1in}

\hspace{-12mm} \parbox{0.55\textwidth}{
\begin{center}
\includegraphics[width=0.55\textwidth]{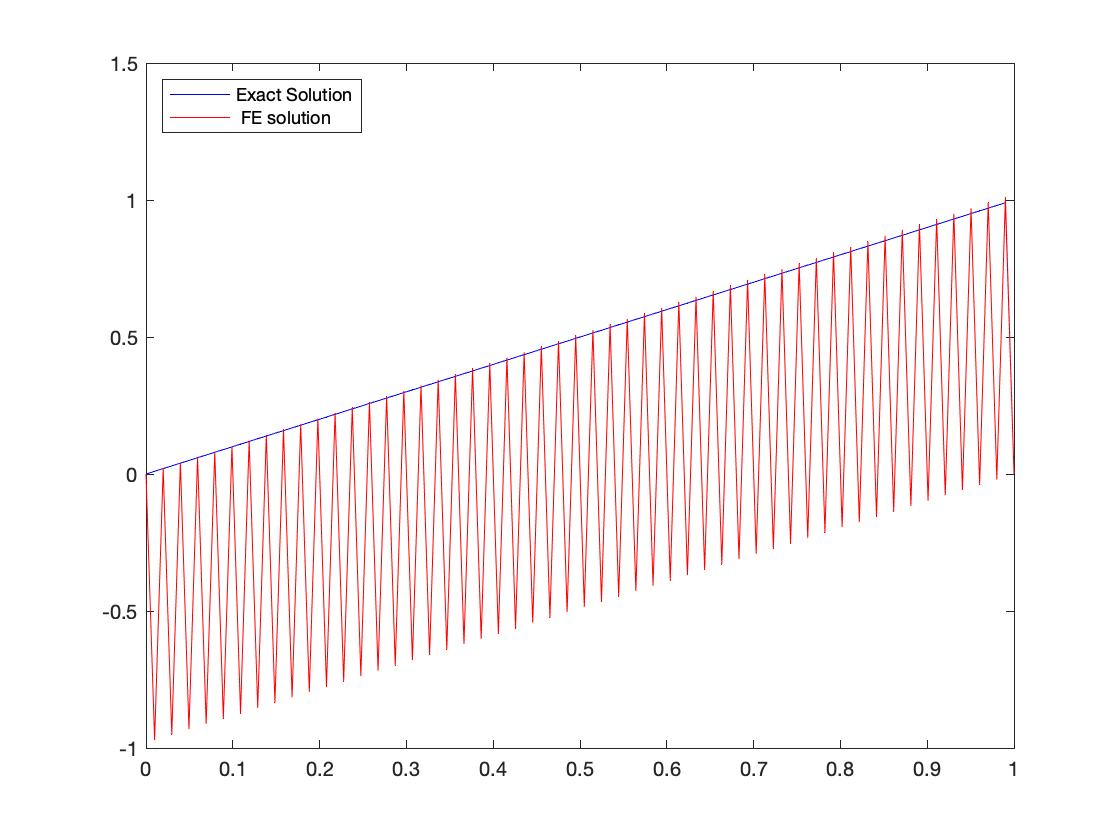}\\
{Fig.1: $f=1, n=101, \varepsilon=10^{-6}$} 
\label{fig:Fig1}
\end{center}
}
\parbox{0.55\textwidth}{
\begin{center}
\includegraphics[width=0.55\textwidth]{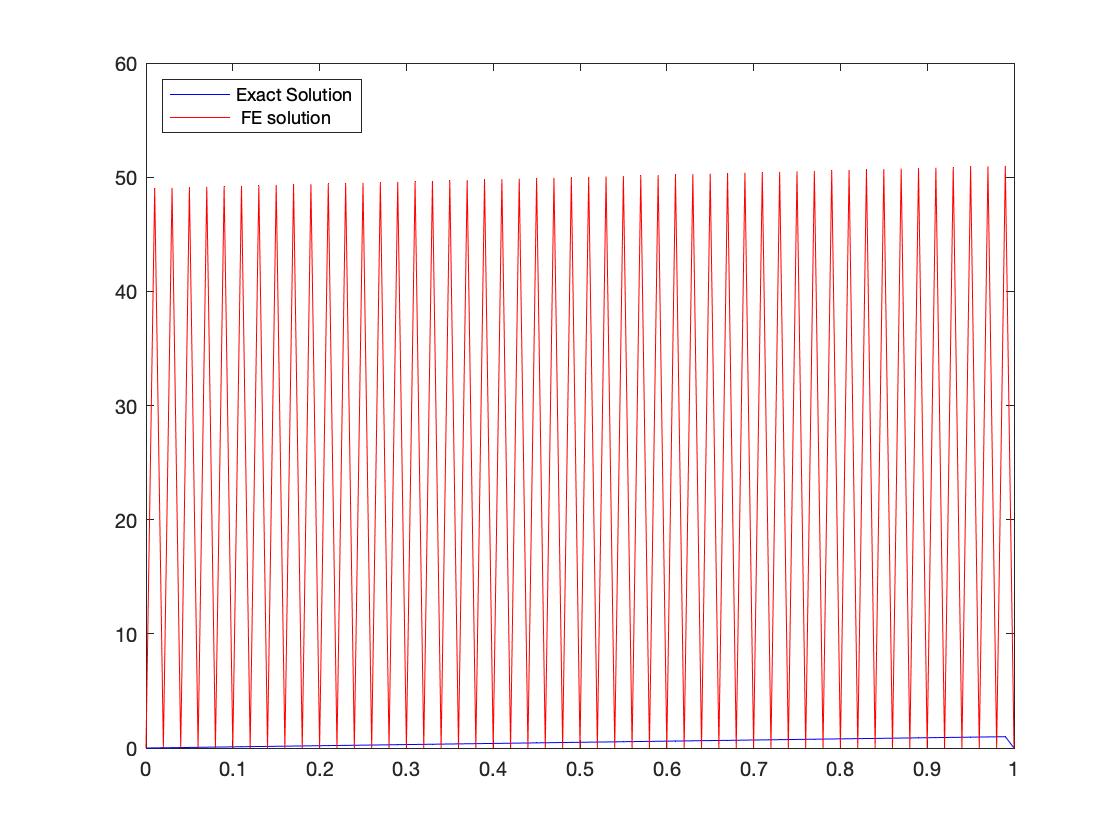}\\
{Fig.2: $f=1, n=102$,  $\varepsilon=10^{-6}$} 
\label{fig:Fig2}
\end{center}
}

\vspace{0.1in}
 
 \hspace{-12mm}\parbox{0.55\textwidth}{
\begin{center}
\includegraphics[width=0.55\textwidth]{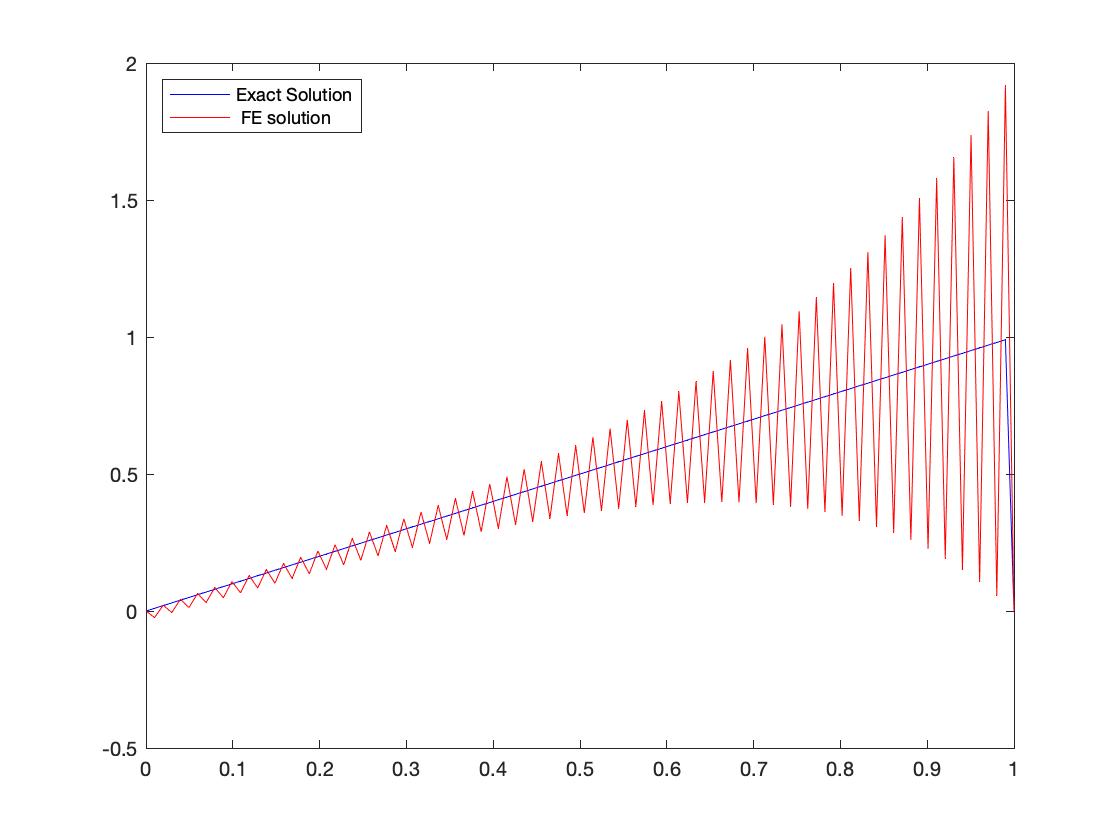}\\
{Fig.3:  $f=1, n=101, \varepsilon=10^{-4}$} 
\end{center}
}
\parbox{0.55\textwidth}{
\begin{center}
\includegraphics[width=0.55\textwidth]{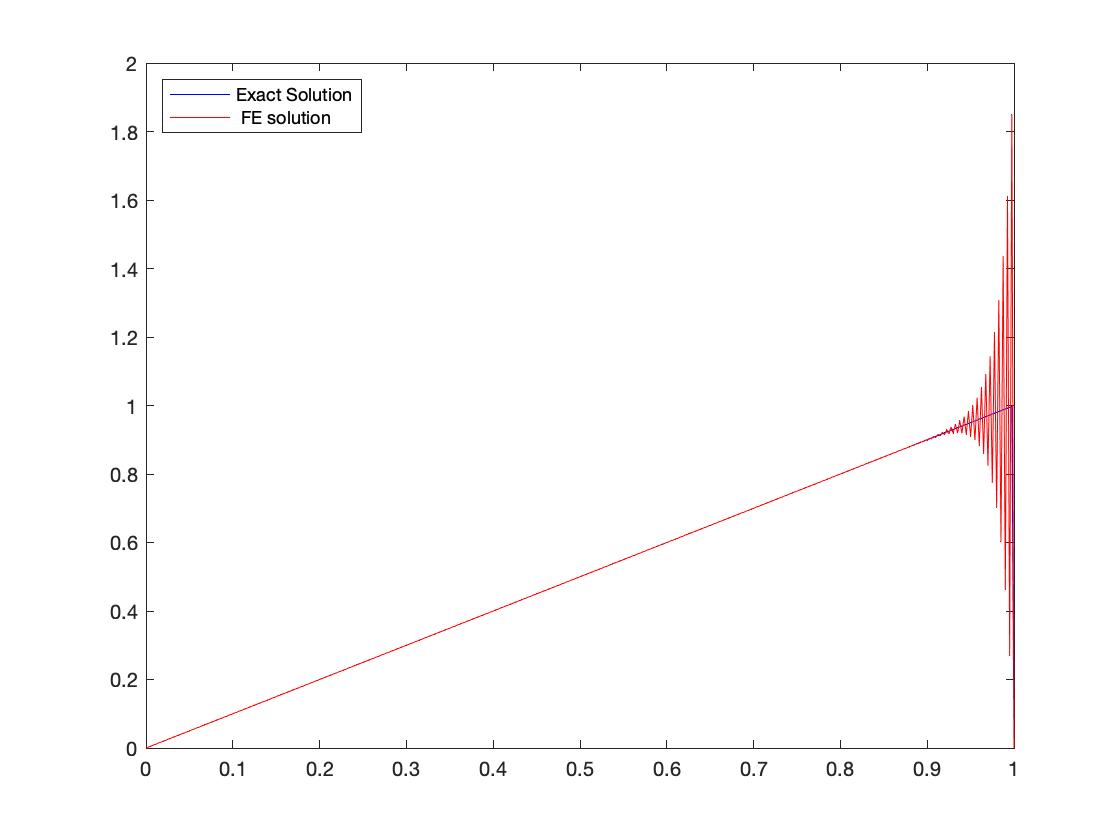}\\
{Fig.4: $f=1, n=400, \varepsilon=10^{-4}$} 
\end{center}
}
\vspace{0.1in}

\hspace{-12mm} \parbox{0.55\textwidth}{
\begin{center}
\includegraphics[width=0.55\textwidth]{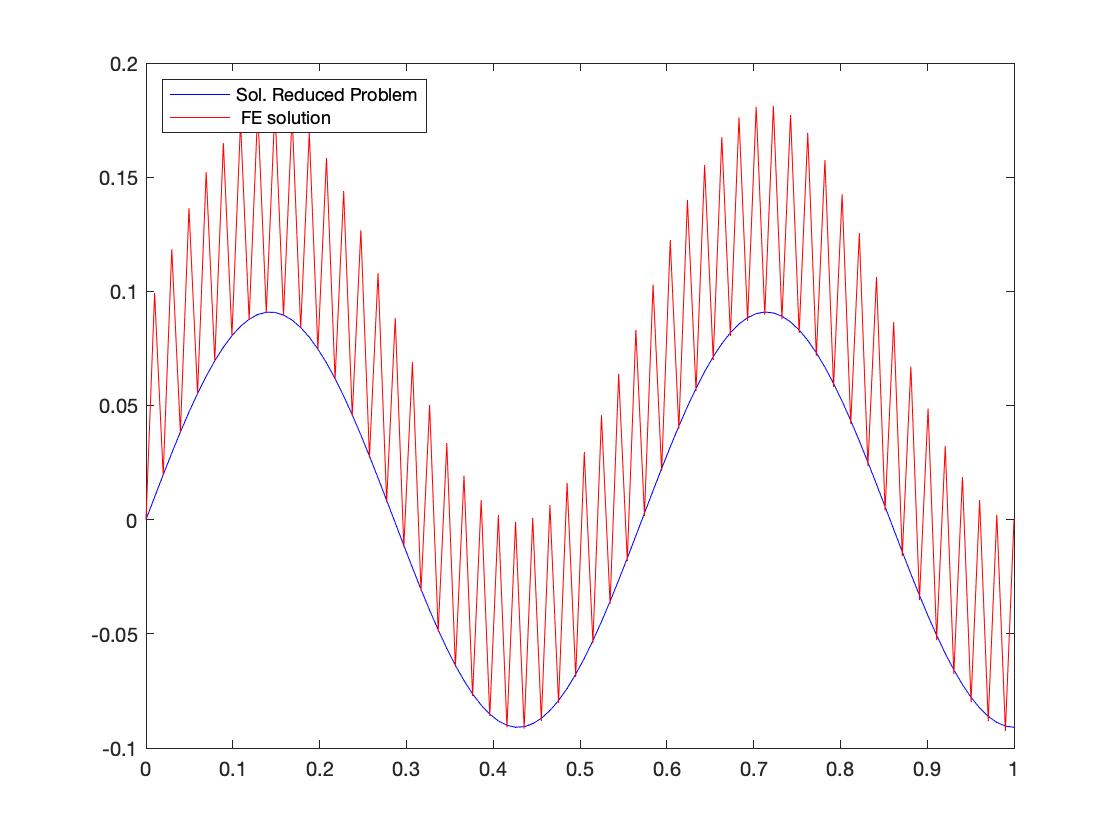}\\
{Fig.5:$f=\cos(\frac{7\pi}{2} x), n=101, \varepsilon=10^{-6}$} 
\end{center}
}
\parbox{0.55\textwidth}{
\begin{center}
\includegraphics[width=0.55\textwidth]{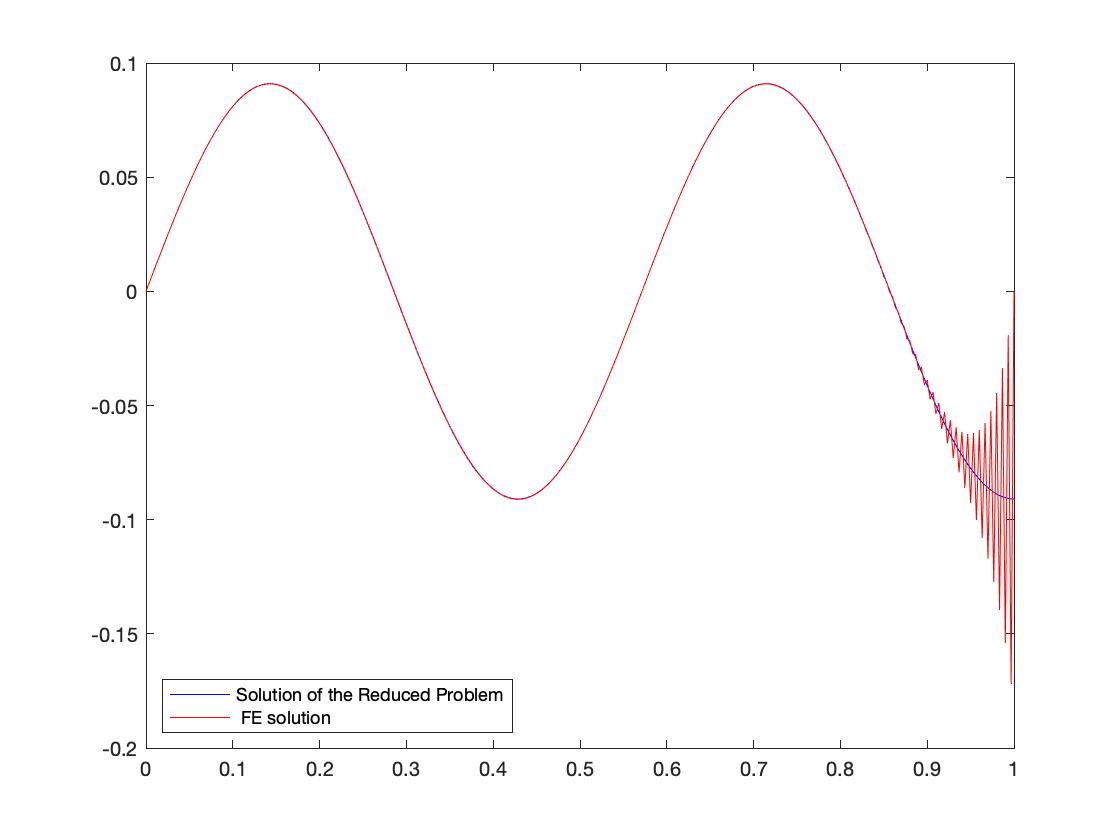}\\
{Fig.6:$ n=300, \varepsilon=10^{-4}$} 
\end{center}
}
\vspace{0.1in}
The solution of \eqref{VF1d-reduced2} is  $\theta(x)= -\int_x^1 f(s)\, ds$. Thus,  $\theta(x)=w(x)- \int_0^1 f(x)\, dx$.

For the case $n=2m$, the system \eqref{sys-even} is identical, but since $u_0=u_{2m} =0$, the system might not have a solution.  In addition, the second system   \eqref{sys-odd} with  the last equation  removed,  is undetermined, and could have infinitely many solutions. The discretization of \eqref{dVF}  is still very oscillatory in this case, see Fig.2.

Numerical  tests for the  case $\int_0^1  f(x)  \, dx\neq 0$ show that, as $\varepsilon/h \to 1$, the linear finite element solution of \eqref{dVF}  oscillates between two curves 
and approximates well the graph of $w$ on intervals $[0, \, \alpha(h)]$ with $\alpha(h)\to 1$  as $h$ gets closer and closer to $\epsilon$, see Fig.3, Fig.4, and Fig.6. 

The behavior of the standard linear finite element approximation  of \eqref{dVF}  motivates  the use of non-standard discretization approaches, such as the  {\it saddle point   least square} or {\it Petrov-Galerkin} methods.

\subsection{SPLS discretization} \label{sec:MMD}
For improving the stability  and approximability of the finite element approximation a  {\it  saddle point least square}  (SPLS) method has  been used, see e.g., \cite{BM12,Dahmen-Welper-Cohen12, dem-fuh-heu-tia19}.  The SPLS method for solving \eqref{VF1d} is: Find $(w, u) \in V \times Q$ such that 
\begin{equation}\label{SPLS4model2}
\begin{array}{lclll}
a_0(w,v) & + & b( v, u) &= (f,v ) &\ \Forall  v \in V,\\
b(w,q) & & & =0   &\  \Forall  q \in Q,   
\end{array}
\end{equation}
where $V=Q= H^1_0(0,1)$, with possible  different type of norms, and \\  $b(v, u) =\varepsilon\, a_0(u, v)+(u',v)= \varepsilon\, (u', v')+(u',v)$. 

For the discretization of \eqref{SPLS4model2},  we choose finite element spaces
$\M_h \subset Q$ and $V_h \subset V$ and solve the discrete problem: 
Find $(w_h, u_h) \in V_h \times \M_h$ such that 
\begin{equation}\label{SPLS4model-h}
\begin{array}{lclll}
a_0(w_h,v_h) & + & b( v_h, u_h) &= (f,v_h ) &\ \Forall  v_h \in V_h,\\
b(w_h,q_h) & & & =0   &\  \Forall  q_h \in \M_h.   
\end{array}
\end{equation}


Similar analysis and numerical results for finite element  test and trail spaces of various degree polynomial were done in  \cite{Dahmen-Welper-Cohen12, dem-fuh-heu-tia19}. In this section, we provide some numerical results for  $\M_h= C^0-P^1:= span\{ \varphi_j\}_{j = 1}^{n-1}$, with $\varphi_j$'s  the standard linear nodal functions  and $V_h=C^0-P^2$  on the given uniformly distributed nodes on $[0, 1]$, to show the  improvement from the standard linear discretization. We note that, using the  optimal norm on $\M_h$ (see Section \ref{sec:analysis-quad}),  we have a discrete $\inf-\sup$ condition satisfied. The presence of non-physical oscillation is diminished, and the  errors are better for the SPLS discretization, see Table 1 and Table 2.  

For  $\int_0^1f(x)\, dx=0$ there is no much difference in the solution behaviour for the two methods.  
But, for  $\int_0^1  f(x) \, dx\neq 0$, our numerical tests  showed  an essential improvement for the SPLS solution. Inside the interval $[3h, 1-3h]$ the SPLS solution $u_h$, approximates the  shift by a constant  of  the solution $u$ of the original problem \eqref{VF1d},   see Fig.7-Fig.10. The oscillations appear only at the ends of the interval. 
The behavior can be explained by similar arguments presented in Section \ref{sec:LinP1} as follows: 
The {\it simplified} problem, obtained from \eqref{SPLS4model2} by letting $\varepsilon \to 0$, is not well posed when $\int_0^1  f(x)  \, dx\neq 0$. However, the {\it simplified} linear system obtained from \eqref{SPLS4model-h}  by letting $\varepsilon \to 0$, i.e.: Find $(w_h, u_h) \in V_h \times \M_h$ such that 
\begin{equation}\label{SPLS4model-h-R}
\begin{array}{lclll}
(w'_h,v'_h) & + & (u'_h, v_h) &= (f,v_h ) &\ \Forall  v_h \in V_h,\\
(w_h,q'_h) & & & =0   &\  \Forall  q_h \in \M_h,   
\end{array}
\end{equation}
 has unique solution, because a discrete $\inf-\sup$ condition, using optimal trial norm, holds (see Section \ref{sec:analysis-quad}).    Numerical tests  for $\varepsilon \leq 10^{-3}$ show that the solution of the simplified system \eqref{SPLS4model-h-R} approximates  the function  $\frac{1}{2} (w(x) +\theta(x))$ where $w, \theta$, are  the solutions of the reduced problems \eqref{VF1d-reduced} and \eqref{VF1d-reduced2}. Similar type of  oscillations depending only on $h$ towards the ends of $[0, 1]$ are still present. For example, for $f=1$  and $n=101$, the solution of \eqref{SPLS4model-h-R} is  close to $x-1/2$, see  Fig.7. 
For  $\varepsilon/h \leq 10^{-4}$ the solution of \eqref{SPLS4model-h} is close to the solution of   \eqref{SPLS4model-h-R}.  However, as $10^{-4}<\varepsilon/h  \to 1$, the  solution of  \eqref{SPLS4model-h} is decreasing  the size of the shifting constant and approximates  $u$, rather than ${1}/{2} (w(x) +\theta(x))$.  Similar oscillations are still present, but only outside of the interval $[3h, 1-3h]$. The error analysis of Section \ref{sec:analysis-quad} reveals the  discrete solution behavior based on the explicit form we find for the optimal norm on $\M_h$. 


\hspace{-12mm} \parbox{0.5\textwidth}{
\begin{center}
\includegraphics[width=2.5in]{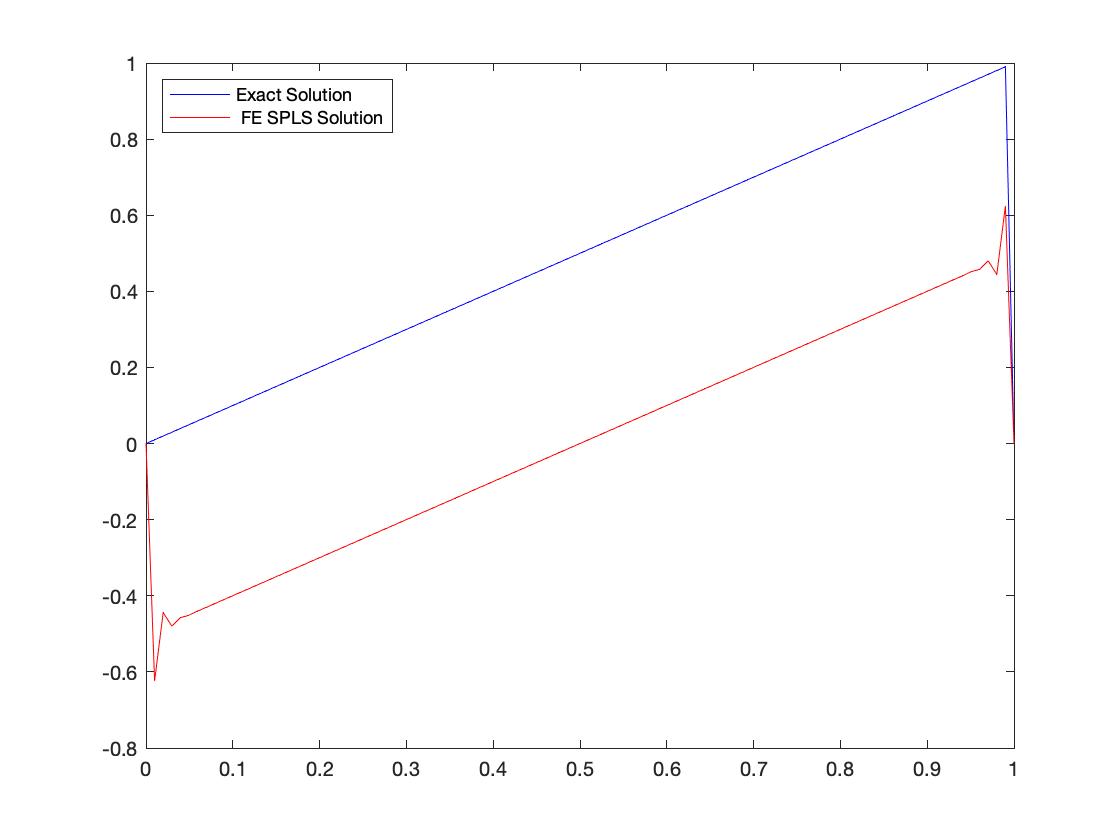}\\
{Fig.7:  $f=1, n=101, \varepsilon=10^{-6}$} 
\end{center}
}
\parbox{0.55\textwidth}{
\begin{center}
\includegraphics[width=0.5\textwidth]{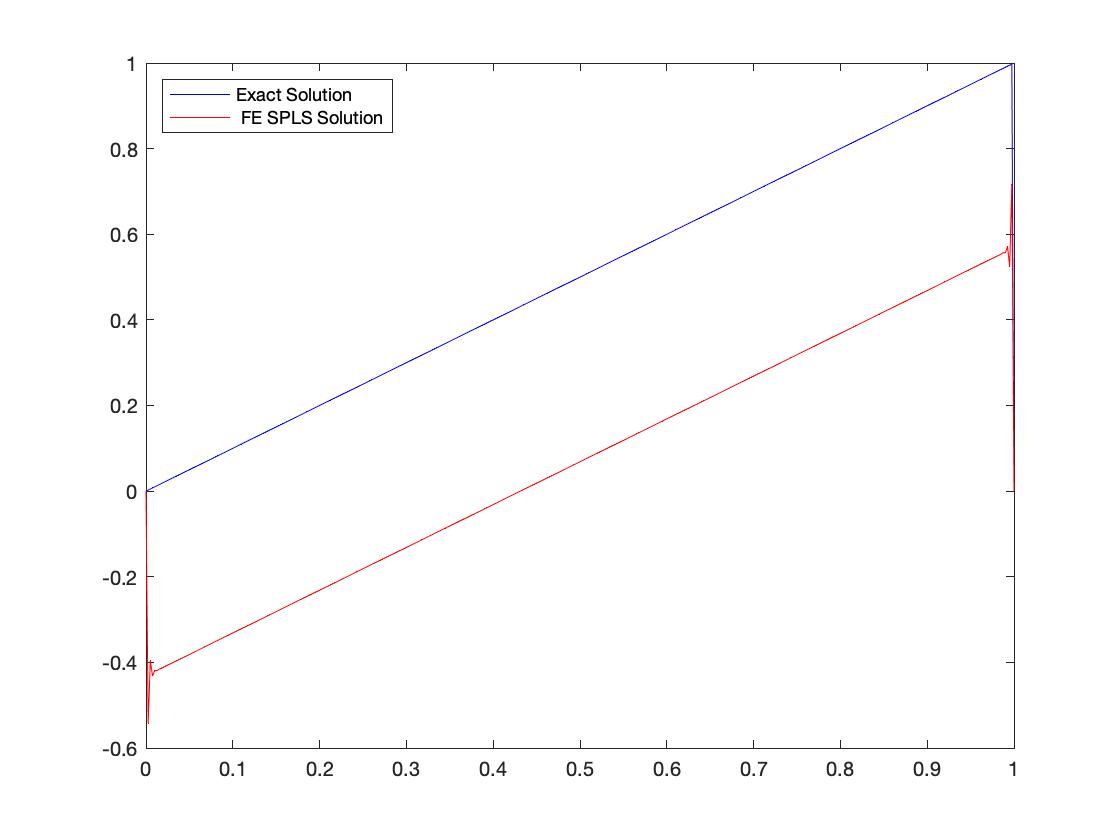}\\
{Fig.8: $f=1, n=400, \varepsilon=10^{-4}$} 
\end{center}
}
\vspace{0.1in}

\hspace{-12mm} \parbox{0.55\textwidth}{
\begin{center}
\includegraphics[width=0.55\textwidth]{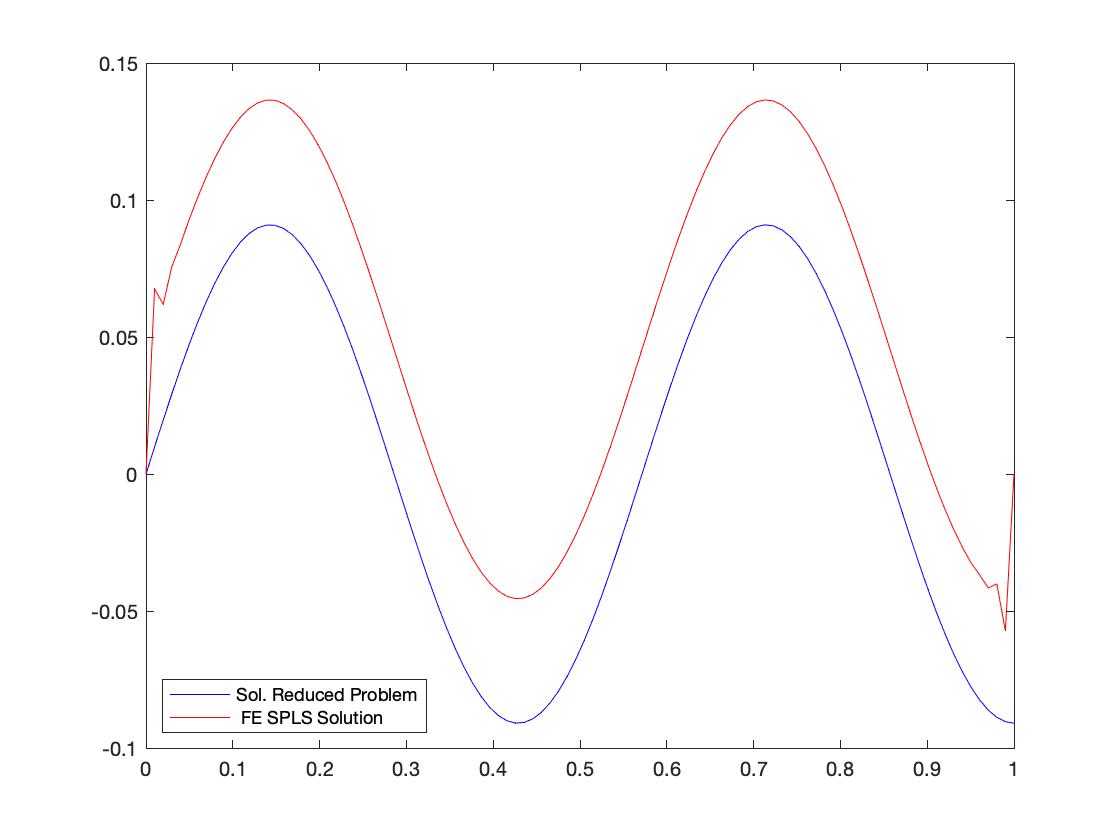}\\
{Fig.9 $f=\cos(\frac{7\pi}{2} x), n=101, \varepsilon=10^{-6}$} 
\end{center}
}
 \parbox{0.55\textwidth}{
\begin{center}
\includegraphics[width=0.55\textwidth]{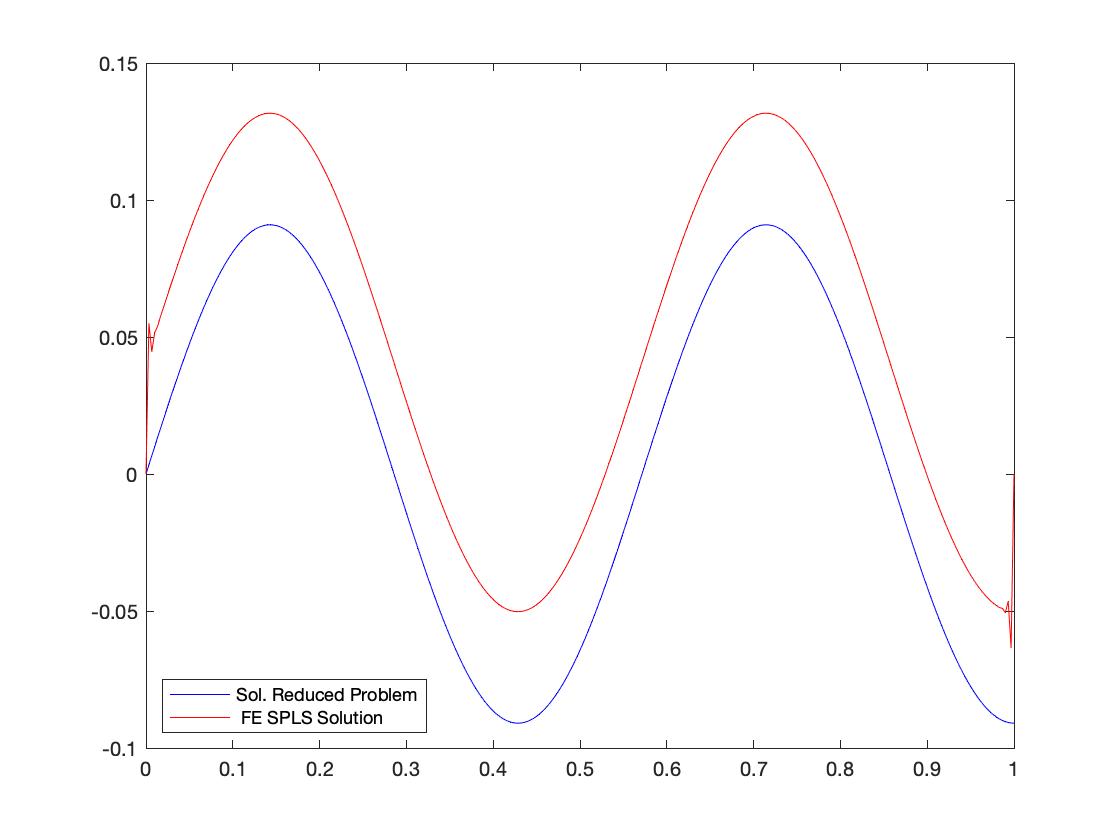}\\
{Fig.10 $f=\cos(\frac{7\pi}{2} x), n=300, \varepsilon=10^{-4}$} 
\end{center}
}
\vspace{0.1in}

\subsection{Petrov Galerkin (PG) with bubble enriched test space $V_h$}\label{ssec:PGbubbles}

We consider   $b(v,u) :=\varepsilon\, a_0(u, v)+(u',v) $  for all $ u, v \in V:=H^{1}_0(0,1)$. We view the PG method as a particular case of the SPLS formulation  \eqref{SPLS4model2}.
The second equation in \eqref{SPLS4model2} implies $w=0$, and the SPLS problem reduces to:  Find $u \in Q$ such that 
\begin{equation}\label{PG4model}
b( v, u) = (f,v ) \ \Forall  v \in V,\\
\end{equation}
which is a Petrov-Galerkin method for solving \eqref{eq:1d-model}.
\subsubsection{Upwinding Petrov Galerkin discretization}
One of the well known  Petrov-Galerkin discretization of the model problem \eqref{PG4model} with $\M_h= span\{ \varphi_j\}_{j = 1}^{n-1}$ consists of modifying the test space such that diffusion is created from the convection term.  This is also known as an {\it upwinding finite element scheme}, see Section 2.2 in \cite{roos-stynes-tobiska-96}. We define the test space $V_h$ by introducing a  bubble function for each interval $[x_{i-1}, x_i], i=1,2, \cdots, n$:
 \[
B_i:= 4\, \varphi_{i-1}\, \varphi_i, \ \ i=1,2, \cdots, n,
\] 
which is supported in $[x_{i-1}, x_i]$. The discrete test space $V_h$ is 
\[
V_h:= span \{ \varphi_j  + B_{j}-B_{j+1}\}_{j = 1}^{n-1}. 
\]
We note that both $\M_h$ and $V_h$ have dimension $n-1$ and $V_h\subset C^0-P^2$. 

In a more general approach the test functions can be defined using  upwinding parameters $\sigma_i >0$ to get $V_h:= span \{ \varphi_j  + \sigma_i( B_{j}-B_{j+1})\} _{j = 1}^{n-1}$. 

\subsubsection{Variational formulation  and matrices}
The upwinding Petrov Galerkin discretization for 
\eqref{eq:1d-model} is: Find $u_h \in \M_h$ such that 
\begin{equation}\label{PG4model-h}
b( v_h, u_h) = (f,v_h) \ \Forall  v_h \in V_h. 
\end{equation}
We look for 
\[
u_h= \sum_{j=1}^{n-1} \alpha_j \varphi_j,
\]
and consider a generic test function 
\[
v_h= \sum_{i=1}^{n-1} \beta_i \varphi_i + \sum_{i=1}^{n-1}  \beta_i (B_i - B_{i+1}) = \sum_{i=1}^{n-1} \beta_i \varphi_i + \sum_{i=1}^{n}  (\beta_i - \beta_{i-1}) B_{i},
\]
where, we define $\beta_0=\beta_n=0$. Denoting,
\[
B_h:=\sum_{i=1}^{n}  (\beta_i - \beta_{i-1}) B_{i},  \ \text{and } \  w_h:=  \sum_{i=1}^{n-1} \beta_i \varphi_i ,
\]
we have 
\[
v_h=w_h + B_h.
\]
For a generic bubble function $B$  with support $[a, b]$, we have
\[
B:= \frac{4}{(b-a)^2} (x-a) (b-x), \ \text{with}  \ a<b,\  \text{and}
\]
\begin{equation} \label{eq:B-prop}
\int_a^b B(x)\, dx  = \frac{2 (b-a)}{3}, \ 
\int_a^b B'\, dx  =0, \ 
\int_a^b (B')^2\, dx  =\frac{16}{3 (b-a)}.
\end{equation}
Using the above formulas, the fact that $u'_h, w'_h $  are constant  on each of  the intervals $[x_{i-1}, x_i]$,  and that $w'_h= \frac{\beta_i -\beta_{i-1} }{h}$ on $[x_{i-1}, x_i]$,  we obtain
\[
(u'_h, B_h) = \sum_{i=1}^{n}\int_{x_{i-1}}^{x_i}  u'_h (\beta_i - \beta_{i-1}) B_{i}=
 \sum_{i=1}^{n} u'_h \,  w'_h \int_{x_{i-1}}^{x_i}  B_{i} =\frac{2h}{3}  \sum_{i=1}^{n} \int_{x_{i-1}}^{x_i}  u'_h  w'_h. 
\]
Thus
\begin{equation} \label{eq:upBh}
(u'_h, B_h) =\frac{2h}{3}  (u'_h, w'_h), \ \text{where} \  v_h=w_h + B_h.
\end{equation}

In addition, 
\[
(u'_h, B'_i) =0 \ \text{for all} \ i=1, 2, \cdots,  n, \text{hence} 
\]
\begin{equation} \label{eq:upBph}
 (u'_h, B'_h) =0,  \text{for all} \  u_h \in \M_h, v_h=w_h + B_h \in V_h.
\end{equation}
From  \eqref{eq:upBh} and  \eqref{eq:upBph}, for any $u_h \in \M_h, v_h=w_h + B_h \in V_h$ we get
\begin{equation} \label{eq:bPG}
 b(v_h, u_h) = \left (\varepsilon + \frac{2h}{3}\right )  (u'_h, w'_h) +  (u'_h, w_h).
\end{equation}
Thus, adding the bubble part to the test space leads to the extra diffusion term $  \frac{2h}{3}  (u'_h, w'_h)$ with  $\frac{2h}{3} >0$  matching the sign of the coefficient of $u'$ in  \eqref{eq:1d-model}. It is also interesting to note  that  only the linear part of $v_h$ appears in the expression of $ b(v_h, u_h)$.  The  functional  $v_h \to (f, v_h)$ can  be also viewed as a functional only of the linear part $w_h$. Indeed, using the splitting $v_h=w_h + B_h $  and that $B_h:=\sum_{i=1}^{n}  (\beta_i - \beta_{i-1}) B_{i}$, we get
\[
(f, v_h) =(f, w_h) + (f, \sum_{i=1}^{n}  h w'_h B_i) =(f, w_h) +h\,  (f, w'_h  \sum_{i=1}^{n}   B_i). 
\]
The variational formulation of the upwinding  Petrov-Galerkin method can be reformulated as: Find $u_h \in \M_h$ such that 
\begin{equation}\label{PG4model-hR}
 \left (\varepsilon + \frac{2h}{3}\right )  (u'_h, w'_h) +  (u'_h, w_h) = (f, w_h )  + h\,  (f, w'_h  \sum_{i=1}^{n}   B_i),  w_h \in M_h. 
\end{equation}
 The reformulation allows for a new error analysis using an optimal test norm and facilitates the  comparison with the known {\it stream-line diffusion} (SD) method  of discretization, as presented in the next section. 
 
 For the analysis of the method, using \eqref{eq:upBph} and the last part of \eqref{eq:B-prop},  we  note that for any $v_h=w_h + B_h \in V_h$ we have 
 \[
 \begin{aligned}
 (v'_h, v'_h) & =(w'_h + B'_h, w'_h + B'_h) = (w'_h, w'_h) + ( B'_h, B'_h) =\\ 
 & =(w'_h, w'_h) +  \sum_{i=1}^{n} (\beta_i-\beta_{i-1})^2 (B'_i, B'_i)=\\
 & =(w'_h, w'_h)   + \frac{16h}{3}  \sum_{i=1}^{n} \left (\frac{\beta_i-\beta_{i-1}}{h} \right )^2 = \\
 & =(w'_h, w'_h)   +  \frac{16}{3}  \sum_{i=1}^{n} \left (\int_{x_{i-1}}^{x_i} (w'_h)^2 \right )^2 =(w'_h, w'_h)   +  \frac{16}{3} (w'_h, w'_h). 
 \end{aligned}
 \]
 
Consequently,
\begin{equation} \label{eq:vhwh}
|v_h|^2 = \frac{19}{3} |w_h|^2. 
\end{equation}
Using the reformulation \eqref{PG4model-hR}, the  linear system  to be solved is 
\begin{equation}\label{1d-PG-ls}
\left ( \left (\frac{\varepsilon}{h} +  \frac{2h}{3} \right ) S+ C \right )\, U = F_{PG}, 
\end{equation}
where \(U,F_{PG}\in\R^{n-1}\)  with:
 \[
 U:=\begin{bmatrix}u_1\\u_2\\\vdots\\u_{n-1}\end{bmatrix},\quad F_{PG}:= \begin{bmatrix}(f,\varphi_1)\\ (f,\varphi_2)\\\vdots \\ (f,\varphi_{n-1})\end{bmatrix} + 
 \begin{bmatrix} (f, B_1 -B_2) 
 \\  (f, B_2-B_3)\\ \vdots \\  (f, B_{n-1} -B_n) \end{bmatrix}, 
\]
and $S, C$ are the matrices defined at the beginning of this section.
Numerical tests show that this method does not lead to any kind of non-physical oscillations. We will provide our analysis of the method as a mixed method in Section \ref{ssec:PGanalysis}. 

\subsection{Stream line diffusion (SD)  discretization} \label{ssec:SD} 
Classical ways to introduce the SD method can be found in e.g., \cite{brezzi-marini-russo98,hughes-brooks79}. For our model problem, we  relate and compare the method with the upwinding PG method. 
We take  $\M_h= V_h= span\{ \varphi_j\}_{j = 1}^{n-1}$ and consider the  {\it stream line diffusion method} for solving  
\eqref{eq:1d-model}: Find $u_h \in \M_h$ such that 
\begin{equation}\label{SD4model-h}
b_{sd} (w_h, u_h) = F_{sd}(w_h) \ \Forall  w_h \in V_h, 
\end{equation} 
where 
\[
b_{sd} (w_h, u_h):= \varepsilon\, (u'_h, w'_h)+(u'_h,w_h) +  \sum_{i=1}^{n} \delta_i \int_{x_{i-1}}^{x_i} u'_h w'_h
\]
with $\delta_i >0$ weight parameters, and 
\[
F_{sd}(w_h) := (f, w_h) + \sum_{i=1}^{n} \delta_i \int_{x_{i-1}}^{x_i} f(x)\,  w'_h\, dx. 
\]
In a more general approach, $\delta_i$'s  are chosen  as functions of  $x_i-x_{i-1}$. Optimal choices for $\delta_i$'s are discussed in e.g., \cite{chen-xu08, linssT10}. For the choice 
\[
\delta_i = \frac{2h}{3}, \  i=1,2,\cdots, n,
\]
and arbitrary  $w_h, u_h \in \M_h =V_h$ the bilinear form $b_{sd}$ becomes 
\[
b_{sd} (w_h, u_h)=b(w_h,u_h) =\left (\varepsilon + \frac{2h}{3}\right )  (u'_h, w'_h) +  (u'_h, w_h),
\]
and the corresponding right hand side  functional  $F_{sd}$ is 
\begin{equation}\label{eq:RHS-sd}
F_{sd}(w_h)  = (f, w_h) + \frac{2h}{3}(f,w'_h), \ w_h \in V_h. 
\end{equation} 
Thus, by choosing the appropriate  weights, the upwinding PG  and SD discretization methods  lead to the same stiffness matrix.  By comparing the right hand sides of 
\eqref{PG4model-hR} and \eqref{eq:RHS-sd}, we note that the two methods produce the same system (solution)  if and only if 
\begin{equation}\label{eq:RHS=}
(f, w'_h  \sum_{i=1}^{n}   B_i)=  \frac{2}{3}(f,w'_h), \ \text{for all} \ w_h \in V_h=C^0-P^1.
\end{equation} 
This is a  feasible condition, as 
\[
 \int_0^1 \sum_{i=1}^{n}   B_i = n\, \frac{2h}{3}= \frac{2}{3}.
 \]
In fact, the condition \eqref{eq:RHS=}  is satisfied for $f=1$. In this case, both sides of \eqref{eq:RHS=}   are zero. 
Due to  a reformulation as a mixed conforming variational  method,  
 we expect the upwinding PG method  to perform better for certain error norms, see Tables \ref{table:Test2},  \ref{table:Test3}, and \ref{table:Test4}. It is known, \cite{ bartels16, quarteroni-sacco-saleri07,roos-stynes-tobiska-96} that the error estimate for the SD method is defined using a special SD-norm. In the one dimensional case with  same  weights $\delta_i =\delta$, the norm becomes
\[
\|v\|^2_{sd} = \varepsilon |v|^2 + \delta  |v|^2. 
\]
For a fair comparison with the PG method, we take $\delta = \frac{2h}{3}$.  Provided the continuous solution $u$ of 
\eqref{eq:1d-model} satisfies $u\in H^2(0,1)$,  for  the SD discrete solution $u_h$ of \eqref{SD4model-h}, we have 

\begin{equation}\label{eq:EE-sd}
\|u-u_h\|_{sd} \leq  c_{sd}\,  h^{3/2} \|u''\|.
\end{equation}

For the comparison of the implementation of the two methods, we  can also compare the load vector $F_{PG}$ defined above to the load vector for the SD method:
\[
F_{SD}:= \begin{bmatrix}(f,\varphi_1)\\ (f,\varphi_2)\\\vdots \\ (f,\varphi_{n-1})\end{bmatrix} + 
 \frac{2h}{3} \begin{bmatrix} (f, \varphi'_1) 
 \\  (f, \varphi'_2)\\ \vdots \\  (f, \varphi'_n) \end{bmatrix}. 
\]

\section{Stability of mixed discretization for the 1D Convection reaction problem} \label{sec:stability}

We consider the discretization of \eqref{1d-LS} with $V=Q=H^1_0(0,1)$. The inner product on $V$ is given by  $a_0(u,v) = (u,v)_V = (u',v')$.  On $Q$ and   $\M_h= span \{ \varphi_j, j=1,2,\cdots,n\}$, we will consider optimal norms that ensure continuous and discrete stability. 

\subsection{Optimal trial norms}\label{ssec:ON}
We define the  anti-symmetric operator \\ $T:Q\to Q$ by 
 \[
 a_0(Tu,q) = (u',q), \ \text{for all} \  q \in Q.
 \]
By solving the corresponding differential equation, one can find that
\begin{equation}\label{eq:T-Action}
Tu = x\overline{u} - \int_0^xu(s)\,ds, 
\end{equation}
and
\begin{equation}\label{eq:T-NormSq}
|Tu|^2 = \int_0^1 |u(s)-\overline{u}|^2\, ds= \|u -\overline{u}\|^2= \|u\|^2 -\overline{u}^2,
\end{equation}
where $\overline{u}=  \int_0^1 u(s)\, ds$. 
 The optimal continuous trial norm on $Q$ is defined by 
\[
 \|u\|_{*}:= \du {\sup} {v \in V}{} \ \frac {b(v, u)}{|v|} =   \du {\sup} {v \in V}{} \ \frac {\varepsilon a_0(u, v) + a_0(Tu,v) }{|v|}.
\]
Using the Riesz representation theorem and the fact that $ a_0(Tu,u) =0$, we obtain that the 
optimal trial norm on $Q$  is given by
\begin{equation}\label{eq:COptNorm}
\|u\|_{*}^2  =\varepsilon^2|u|^2 +|Tu|^2 = \varepsilon^2|u|^2 + \|u\|^2 -\overline{u}^2.
\end{equation}
The advantage of using  the  optimal trial norm on $Q$   resides in the fact that both $inf-sup$ and $sup-sup$ constants  at the continuous level are one. 

For the purpose of obtaining a discrete optimal norm on $\M_h$, we  let $P_h:Q\to V_h$ be the standard elliptic projection defined by 
\[
a_0(P_h\, u, v_h) = a_0(u,v_h), \ \text{for all} \, v_h \in V_h, 
\]
where the discrete test space $V_h$ could be different from $\M_h$.  The optimal trial norm   on $ \M_h$ is 

 \begin{equation}\label{eq:dotn}
  \|u_h\|_{*,h}:= \du {\sup} {v_h \in V_h}{} \ \frac {b(v_h, u_h)}{|v_h|}.
 \end{equation}
As in the continuous case, 
\[
 \|u_h\|_{*,h}:=  \du {\sup} {v_h \in V_h}{} \ \frac {\varepsilon a_0(u_h, v_h) + a_0(Tu_h,v_h) }{|v_h|}  = \du {\sup} {v_h \in V_h}{} \ \frac {\varepsilon a_0(u_h, v_h) + a_0(P_h\,Tu_h,v) }{|v_h|}.
\]
 Assuming that $\M_h\subset V_h$ and using
\[
a_0(P_h\, Tu_h, u_h) = a_0(T u_h,u_h)=0,
\]
by the Riesz representation theorem on $V_h$, we get 

\begin{equation}\label{eq:COptNorm-h}
\|u_h\|_{*,h}^2  =\varepsilon^2|u_h|^2 +|P_hTu_h|^2 := \varepsilon^2|u_h|^2 +|u_h|^2_{*,h}.
\end{equation}
Note that for the given  trial  spaces  $\M_h$  and $Q$, the above norm is well defined for any $u\in Q$. Hence, the continuous and discrete optimal trial norms can be compared on $Q$.

 \subsection{Analysis of  the standard linear discretization with optimal trial norm}\label{sec:analysis-lin} 
 
  We let  $V_h= \M_h= span\{ \varphi_j\}_{j = 1}^{n-1}$, with $\varphi_j$'s the standard linear nodal functions. 
 The optimal trial norm   on $ \M_h$ is given by \eqref{eq:COptNorm-h}. However,  in the one dimensional case, the elliptic projection  $P_h$ on piecewise linears coincides with the interpolant, see e.g., \cite{CB2}. Consequently, using the formula \eqref{eq:T-Action}  for $Tu$, we obtain that  
 $P_h Tu$ can be determined explicitly. Thus,  we have 
 \[
 P_h(Tu)(x_i) = x_i \int_{0}^{1} u(s)\, ds -\int_{0}^{x_i} u(s)\, ds, \ \text{and} 
 \]
 \begin{equation}\label{eq:PhTu-Norm}
|u|^2_{*,h}:= |P_hTu|^2 = \frac{1}{h}\sum_{i=1}^n \left(\int_{x_{i-1}}^{x_i} u(x)\,dx\right)^2 - \left(\int_0^1u(x)\,dx\right)^2.
\end{equation}

 To see more precise estimates that relate the norms  $ \|\cdot \|^2_{*,h}$   and $\|\cdot \|^2_{*}$, we define $c_0$ to be the best constant in the following  Poincare Inequality
\begin{equation}\label{eq:P}
 \|w\| \leq c_p (b-a)\, |w|, \text{ for all }  w \in L^2_0(a,b) \cap H^1(a,b). 
\end{equation}
It has been known that the best constant is $c_p=1/{\pi} $, which can be proved by using the spectral theorem for compact operators on Hilbert spaces for  the inverse of the (1d) Laplace operator with homogeneous Neumann boundary conditions. A more direct  proof for  \eqref{eq:P} can be done  with the  constant $c=1/\sqrt{2}$.
 \begin{prop}\label{relateOptNorms}
Let  $\|u\|_{*}$ and $\|u\|_{*,h}$ be the optimal trial norms defined in Section \ref{ssec:ON} for $V_h= \M_h= span\{ \varphi_j\}_{j = 1}^{n-1}$ . We have 
\begin{equation}\label{eq:eqNorms}
 \|u\|^2_{*,h}  \leq \|u\|^2_{*}\leq  \|u\|^2_{*,h}  + c_p^2 h^2 |u|^2\ \text{for all} \  u\in Q.
\end{equation}
\end{prop}
\begin{proof}
 The left side inequality follows from comparing the formulas for   $\|u\|_{*}$  and  $ \|u\|_{*,h}$, and by using that the norm of the projection operator $P_h$ is one.   For the other inequality, using \eqref{eq:P}  on each subinterval $[x_{i-1}, x_i]$  we  get
  \[
 \begin{aligned}
 \|u\|^2_{*} -\|u\|^2_{*,h} &= \int_0^1u^2(x)\,dx - \frac{1}{h}\sum_{i=1}^n \left( \int_{x_{i-1}}^{x_i} u(x)\,dx\right)^2\\
 &= \sum_{i=1}^n \int_{x_{i-1}}^{x_i} \left (u -\frac{1}{h}  \int_{x_{i-1}}^{x_i} u(x)\,dx \right )^2 \,dx\\
 &\leq c_p^2 h^2 \sum_{i=1}^n \int_{x_{i-1}}^{x_i} (u'(x))^2\,dx = c_p^2 h^2 |u|^2, 
 \end{aligned}
\]
which proves the right side inequality.
 \end{proof}
 From \eqref{eq:eqNorms} it is easy to obtain
 \begin{equation}\label{eq:eqNorms2}
 \|u\|^2_{*,h}  \leq \|u\|^2_{*}\leq \left (1+ \frac{(c_p\, h)^2}{\varepsilon^2} \right ) \|u\|^2_{*,h}\ \text{for all} \  u\in Q.
\end{equation}

As a consequence of the approximation Theorem \ref{th:ap-PG} and \eqref{eq:eqNorms2} we obtain: 

\begin{theorem}\label{th:opt-lin}
If $u$ is the solution of \eqref{VF1d}, and $u_h$ the solution of the  linear discretization \eqref{dVF}, then 
\[
\|u-u_h\|_{*,h}  \leq c(h,\varepsilon) \du{\inf}{v_h \in V_h}{} \ \|u-v_h\|_{*,h}, \ \text{where} 
\]
\[
c(h,\varepsilon)= \sqrt {1+ \frac{(c_p\, h)^2}{\varepsilon^2}  } \approx \frac{c_p\, h}{\varepsilon} \ \text{ if }\ \varepsilon <<h.
\]
\end{theorem}
The estimate can be useful for the case  $\int_0^1 f(x)\, dx=0$, when the $H^1$ or $H^2$ norms of the solutions are less dependent on $\varepsilon$, see e.g. \cite{roos-stynes-tobiska-96}. In this case, we can use that 
\[
\du{\inf}{v_h \in V_h}{} \ \|u-v_h\|_{*,h} \leq c(h,\varepsilon) \|u-u_I\|_{*,h},
\]
where $u_I$  is the linear interpolant of $u$ on the nodes $x_0, x_1, \cdots,x_n$, and exploit the approximation properties of the interpolant. 
In the general case, the error estimate is weak because $\|u\|_{*,h}^2  =\varepsilon^2|u|^2 +|P_hTu|^2$ and $|u |_{*,h}=|P_hTu|$ can be bounded above by $\|u \|_{L^2}$, but is not a norm by itself. Indeed, using  \eqref{eq:PhTu-Norm}  and  the Cauchy-Schwarz inequality, we have 
\begin{equation}\label{weekN}
|u|^2_{*,h} \leq  \frac{1}{n}\sum_{i=1}^n \left(\frac{1}{h} \int_{x_{i-1}}^{x_i} u(x)\,dx\right)^2 \leq 
\frac{1}{n}\sum_{i=1}^n \int_{x_{i-1}}^{x_i} u^2(x)\,dx=\|u\|^2. 
\end{equation}

On the other hand, by rearranging  the integrals in  \eqref{eq:PhTu-Norm}, we obtain 
\[
\begin{aligned}
|u|^2_{*,h} &= \frac{1}{n}\sum_{i=1}^n \left(\frac{1}{h} \int_{x_{i-1}}^{x_i} u(x)\,dx\right)^2 - 
\left( \frac{1}{n} \sum_{i=1}^n \frac{1}{h} \int_{x_{i-1}}^{x_i} u(x)\,dx \right)^2\\
&= \frac{\sum_{i=1}^n \overline{u_i}^2}{n} - \left ( \frac{\sum_{i=1}^n \overline{u_i}}{n}  \right )^2,
\end{aligned} 
\]
where $\overline{u_i}:= \frac{1}{h} \int_{x_{i-1}}^{x_i} u(x)\,dx$. This shows that  $|\cdot |^2_{*,h}$ is in general a seminorm
and we can have $|u|_{*,h} =0$ {\it if and only if}
\[
 \int_{x_{i-1}}^{x_i} u(x)\,dx= \frac{1}{n} \int_0^1u(x)\,dx, \ \text{ for all } i=1,2,\cdots,n.
\]
If $u=w_h\in \M_h$, the above condition can be  satisfied for $n=2m$ and  
$
w_h=C \sum_{i=1}^m \varphi_{2i-1}$, where $C$ is an arbitrary constant.  The graph of  such $w_h$ is highly oscillatory when $h\to 0$ and $|w_h|_{*,h}=0$.  
The ``zig-zag'' behavior of the linear  finite  element solution $u_h$ for small $\varepsilon/h$ can be  justified also   
by noting that $CU=F$ for $U$ the coefficient vector of $w_h$, see  \eqref{1d-LS}. Thus, the solution $u_h$  can capture the oscillatory  mode $w_h$  or is  ``insensitive'' to perturbation by $c\, w_h$.   
\subsection{Analysis of the SPLS with quadratic test space $V_h$} \label{sec:analysis-quad} 

We consider the model problem \eqref{VF1d}  with the discrete space 
 $\M_h= span\{ \varphi_j\}_{j = 1}^{n-1}$ and $V_h:= span \{ \varphi_j \}_{j = 1}^{n-1} + span \{B_{j}\}_{j = 1}^n$, which coincides with  the standard $C^0-P^2$  on the given uniformly distributed nodes on $[0, 1]$. 
In the this section we use the definition  of the discrete trial norm from  Section \ref{ssec:ON}. 
Note that, in this case, the projection $P_h$  is the projection on the space $V_h=C^0-P^2$. For any piecewise linear function $u_h \in \M_h$ we have that 
\[
Tu_h = x\overline{u}_h - \int_0^xu_h(s)\,ds,
\]
is a continuous piecewise quadratic function. Consequently, $Tu_h \in V_h$, and $P_h\, Tu_h =T u_h$. The optimal discrete norm on $\M_h$ becomes 

\[
\|u_h\|_{*,h}^2 =\varepsilon^2 |u_h|^2 + |Tu_h|^2 =\|u_h\|_{*}^2.
\]
Thus, in this case we can consider the same norm given by 
\[
\|u\|_{*}^2 =\varepsilon^2|u|^2 +\|u-\overline{u}\|^2,
\]
 on both spaces $Q$ and $\M_h$. 
As a consequence of the approximation Theorem \ref{th:sharpEE}, we obtain the following  optimal error estimate:

\begin{theorem}
If $u$ is the solution of \eqref{VF1d}, and $u_h$ the  SPLS solution  for the $(P^1-P^2)$ discretization, then  
$$
\|u - u_h\|_{*}\leq \inf_{p_h\in \M_h}\|u - p_h\|_{*} \leq \|u - u_I\|_{*}.
$$
\end{theorem}

\subsection{Analysis of the upwind Petrov Galerkin  method}\label{ssec:PGanalysis}

We consider the model problem \eqref{VF1d} together  with the discrete spaces of Section \ref{ssec:PGbubbles}. Thus, we take 
 $\M_h= span\{ \varphi_j\}_{j = 1}^{n-1}$ and $V_h:= span \{ \varphi_j  + B_{j}-B_{j+1}\}_{j = 1}^{n-1}$. 

Using the formula \eqref {eq:bPG} we can find a representation of  the discrete optimal trial norm. We note that  $b(v_h,u_h)$ can be written  independently  of  the bubble part of $v_h$. This allows us to use the characterization of the discrete trial norm using arguments presented in Section \ref{sec:analysis-lin}. We mention here that the projection $P_h$ used  below is the projection on the $C^0-P^1$ discrete space and not on the test space $V_h$.  
\begin{theorem}
For the discrete problem with continuous piecewise linear trial space $\M_h$ and bubble enriched test space $V_h$, the discrete optimal norm  on $\M_h$ is given by  
\[
\|u_h\|_{*,h}^2 = \frac{3}{19}\left(\left(\varepsilon + \frac{2h}{3}\right)^2|u_h|^2 + |P_hTu_h|^2\right)
\]
where    $|P_hTu_h|=|u_h|_{*,h}$ is given by the formula  \eqref{eq:PhTu-Norm}.
\end{theorem}
\begin{proof}
Using the definition of $\|u_h\|_{*,h}$  along with the work of  Section \ref{ssec:ON} we can reduce the supremum over $V_h$ to a supremum over $\M_h$. Indeed, 
\begin{align*}
\|u_h\|_{*,h}=\sup_{v_h\in V_h}\frac{b(v_h,u_h)}{|v_h|}& = \sup_{w_h\in \M_h}\frac{b(w_h + B_h,u_h)}{|w_h + B_h|}\\
& = \sup_{w_h\in \M_h}\frac{\left(\varepsilon + \frac{2h}{3}\right)(u_h',w_h') + (u_h',w_h)}{\sqrt{\frac{19}{3}}|w_h|}\\
& = \sup_{w_h\in \M_h}\sqrt{\frac{3}{19}}\frac{\left(\left(\varepsilon + \frac{2h}{3}\right)u_h',w_h'\right) + (Tu_h',w_h')}{|w_h|}\\
& = \sup_{w_h\in \M_h}\sqrt{\frac{3}{19}}\frac{\left(\left(\varepsilon + \frac{2h}{3}\right)u_h',w_h'\right) + (P_hTu_h',w_h')}{|w_h|}\\
& = \sqrt{\frac{3}{19}}\left(\left(\varepsilon + \frac{2h}{3}\right)^2|u_h|^2 + |P_hTu_h|^2\right)^{1/2}
\end{align*}
\end{proof}

\begin{prop}\label{propPG}
The following inequality between  $\|u\|_{*}$ and $\|u\|_{*,h}$ holds on  $Q$.
\begin{equation}\label{eq:PG8Ineq}
\|u\|^2_{*} \leq \frac{19}{3} \|u\|^2_{*,h}.
\end{equation}
\end{prop}
\begin{proof}

By using the formula \eqref{eq:COptNorm-h}   and the right side of the inequality \eqref{eq:eqNorms} we have 
\[
\|u\|^2_{*}\leq  \varepsilon^2 |u|^2  +  |P_hTu|^2 +c_p^2 h^2 |u|^2=(\varepsilon^2 + c_p^2 h^2) |u|^2  +  |P_hTu|^2.
\]
To prove \eqref{eq:PG8Ineq}, we just notice that, for $c_p=1/{\pi}$ we have 
\[
(\varepsilon^2 + c_p^2 h^2) |u|^2  +  |P_hTu|^2 \leq  \frac{19}{3} \|u\|^2_{*,h}.
\]

\end{proof}
The following optimal error estimate result is an immediate consequence of Proposition \ref{propPG}.

\begin{theorem}\label{thm:PGError}
If $u$ is the solution of \eqref{VF1d}, and $u_h$ the solution of the upwinding PG formulation, then 
$$
\|u - u_h\|_{*,h}\leq \sqrt{\frac{19}{3}}\inf_{p_h\in \M_h}\|u - p_h\|_{*,h}.
$$
Consequently, 
$$
\|u - u_h\|_{*,h} \leq \varepsilon\mathcal{O}(h) + \mathcal{O}(h^2).
$$
\end{theorem}
\begin{proof}
The first estimate is a direct consequence of the approximation \\ Theorem \ref{th:ap-PG} and the Proposition \ref{propPG}.
For the second estimate, we note that
\[
\begin{aligned}
\|u - u_h\|^2_{*,h} & \leq  {\frac{19}{3}} \inf_{p_h\in \M_h}\|u - p_h\|^2_{*,h}\leq  {\frac{19}{3}}\|u - u_I\|^2_{*,h}\\
& \leq (\varepsilon + {2h}/{3})^2|u-u_I|^2 + |P_hT(u-u_I)|^2
\end{aligned}
\]
where $u_I$ is the linear interpolant of $u$. 
Using \eqref{weekN}, we have that 
\[
|P_hT(u-u_I)|^2 \leq \|u-u_I\|^2,
\]
which leads to 
\[
\|u - u_h\|^2_{*,h}  \leq  (\varepsilon + {2h}/{3})^2|u-u_I|^2 + \|u-u_I\|^2.
\]

For  $u \in H^2(0,1)$,  we have $ \|u-u_I\|_{L^2} = \mathcal{O}(h^2)$, and $|u-u_I|= \mathcal{O}(h)$, and using 
$
(\varepsilon + \frac{2h}{3})^2\leq 2(\varepsilon^2 + h^2),
$
we obtain  
\[
\begin{aligned}
(\varepsilon + {2h}/{3})^2|u-u_I|^2 + \|u-u_I\|^2 & 
\leq 2\left(\varepsilon^2|u-u_I|^2 + h^2|u-u_I|^2\right) +  \mathcal{O}(h^4)\\
& \leq  \varepsilon^2\mathcal{O}(h^2) + \mathcal{O}(h^4),
\end{aligned}
\]
which proves the required estimate. 
\end{proof}
As a consequence of the theorem, we have
\[
(\varepsilon + {2h}/{3})\, |u-u_h| + |P_hT(u-u_h)| \leq  \varepsilon\mathcal{O}(h) + \mathcal{O}(h^2).
\]
We can compare this estimate  with the  error estimate for the SD method, \eqref{eq:EE-sd} : 
\[
(\varepsilon + {2h}/{3})^{1/2} \,  |u-u_h| \leq   \mathcal{O}(h^{3/2}).
\]
Note that for $\varepsilon < < h$ the PG estimate for $|u-u_h|$ is slightly better than the SD estimate, and for   $\varepsilon \approx h$, both estimates lead  to $|u-u_h|  \leq \mathcal{O}(h)$. In addition,  
the PG estimate provides further information about $|P_hT(u-u_h)|$ which, according to \eqref{weekN}, is slightly weaker than the $L^2$ norm  $\|u-u_h\|$. 


\section{Numerical experiments}\label{sec:NR} 

We  compared numerically the standard linear finite element with the $(P^1-P^2)$-SPLS formulation, and the Streamline Diffusion with the upwinding Petrov-Galerkin in a variety of norms.  For the data tables,  we will use the notation $E_{j,method}$ where   $j = 0$  represents  the $L^2$ error $||u-u_h||$, and $j=1$  represents the $H^1$ error $|u-u_h|$. For  {\it method}, we use the following: $L$ for standard piecewise linear approximation; $S$ for the SPLS method; and $SD$ for the Streamline Diffusion method, and $P$ for the  Petrov-Galerkin method. 

\subsection{Standard linear  versus SPLS discretization}\label{ssec:NRlin} 
For the first test, we take $f = 1-2x$ which satisfies the condition $\overline{f}=0$. In this case, we  compare the standard linear finite element method and the SPLS formulation  for two values of $\varepsilon$ that are at least 2 orders of magnitude smaller than $h$ at the finest level. Table \ref{table:Test1} contains the errors of the two methods over six refinements where $h_i = 2^{-i-5}$, for $i=1,2,3,4,5,6$. We can see that for this problem, both methods perform well, however at all levels, for both values of $\varepsilon$, SPLS produces smaller errors. 

\begin{table}[h!] 
\begin{center} 
\begin{tabular}{|*{5}{c|}} 
\hline 
\multicolumn{1}{|c|}{\multirow{2}{*}{\parbox{1.2cm}{\centering Level/$\varepsilon$}}}& \multicolumn{4}{c|}{$10^{-6}$}\\
\cline{2-5} 
&$E_{1,L}$ &$E_{1,S}$&$E_{0,L}$& $E_{0,S}$\\ \hline 
1&0.289&0.144&0.046&0.011\\ 
\hline 
2&0.144&0.072&0.011&0.003\\ 
\hline 
3&0.072&0.036&0.003&0.001\\ 
\hline 
4&0.036&0.018&0.001&1.8e-4\\ 
\hline 
5&0.018&0.009&1.7e-4&4.4e-5\\ 
\hline 
6&0.009&0.005&4.4e-5&1.0e-5\\ 
\hline 
Order & 1& 1& 2& 2\\
\hline
\multicolumn{1}{|c|}{\multirow{2}{*}{\parbox{1.2cm}{\centering Level/$\varepsilon$}}}& \multicolumn{4}{c|}{$10^{-10}$}\\
\cline{2-5} 
&$E_{1,L}$ &$E_{1,S}$&$E_{0,L}$& $E_{0,S}$\\ \hline 
1&0.289&0.144&0.046&0.011\\ 
\hline 
2&0.144&0.072&0.011&0.003\\ 
\hline 
3&0.072&0.036&0.003&0.001\\ 
\hline 
4&0.036&0.018&0.001&1.8e-4\\ 
\hline 
5&0.018&0.009&1.8e-4&4.5e-5\\ 
\hline 
6&0.009&0.005&4.5e-5&1.1e-5\\ 
\hline 
Order &1 &1 & 2& 2\\
\hline
\end{tabular} 
\caption{L vs. SPLS: $f(x) = 1-2x$} 
\label{table:Test1}
\end{center} 
\end{table}

Table \ref{table:Test1.5} contains errors for standard linear finite elements and SPLS for $f(x) = 2x$. As this choice of right hand side does not satisfy the condition that $\overline{f} = 0$, we can expect the results to be less impressive than those of Table \ref{table:Test1}. In Table \ref{table:Test1.5}, we can see that the magnitude of the errors for the two methods is drastically different as $\varepsilon$ decreases. The SPLS method does a better job at capturing the general behavior of the true solution than standard linear finite elements. 

The plots of the approximations are shown at the sixth level of refinement with $\varepsilon = 10^{-6}$ in Figure \ref{fig:LinearvsSPLS}. It can be seen in the plots that the non-physical oscillations of the standard linear elements is present, whereas the SPLS approximation captures the behavior of the true solution, with a shift of the approximation by a constant, invisible for the seminorm part of the optimal discrete norm. Ways on how to modify the discrete spaces in order to eliminate the constant shift for the SPLS approach, remain to be investigated.

\begin{table}[h!] 
\begin{center} 
\begin{tabular}{|*{5}{c|}} 
\hline 
\multicolumn{1}{|c|}{\multirow{2}{*}{\parbox{1.2cm}{\centering Level/$\varepsilon$}}}& \multicolumn{4}{c|}{$10^{-6}$}\\
\cline{2-5} 
&$||u-u_{h,L}||_{*,h}$&Order&$||u-u_{h,S}||_{*,h}$&Order\\ 
\hline 
1&3.52e+01&0.00&4.75e-02&0.00\\ 
\hline 
2&8.81e+00&2.00&3.36e-02&0.50\\ 
\hline 
3&2.21e+00&2.00&2.37e-02&0.50\\ 
\hline 
4&5.75e-01&1.94&1.68e-02&0.50\\ 
\hline 
5&2.02e-01&1.51&1.19e-02&0.50\\ 
\hline 
6&9.97e-02&1.02&8.40e-03&0.50\\ 
\hline 
\multicolumn{1}{|c|}{\multirow{2}{*}{\parbox{1.2cm}{\centering Level/$\varepsilon$}}}& \multicolumn{4}{c|}{$10^{-10}$}\\
\cline{2-5} 
&$||u-u_{h,L}||_{*,h}$&Order&$||u-u_{h,S}||_{*,h}$&Order\\ 
\hline 
1&3.52e+05&0.00&4.75e-02&0.00\\ 
\hline 
2&8.81e+04&2.00&3.36e-02&0.50\\ 
\hline 
3&2.20e+04&2.00&2.37e-02&0.50\\ 
\hline 
4&5.51e+03&2.00&1.68e-02&0.50\\ 
\hline 
5&1.38e+03&2.00&1.19e-02&0.50\\ 
\hline 
6&3.44e+02&2.00&8.40e-03&0.50\\ 
\hline 
\end{tabular} 
\caption{L vs. SPLS: $f(x) = 2x$} 
\label{table:Test1.5}
\end{center} 
\end{table}

\begin{figure}[h!]
\begin{center}
\includegraphics[width=0.5\textwidth]{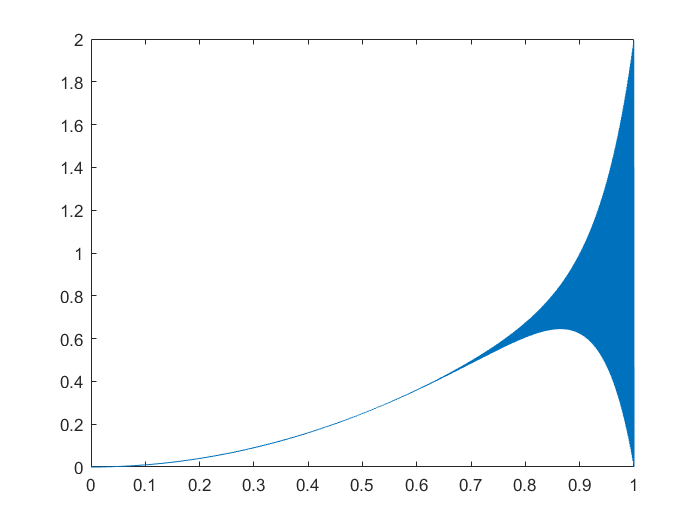}\includegraphics[width=0.5\textwidth]{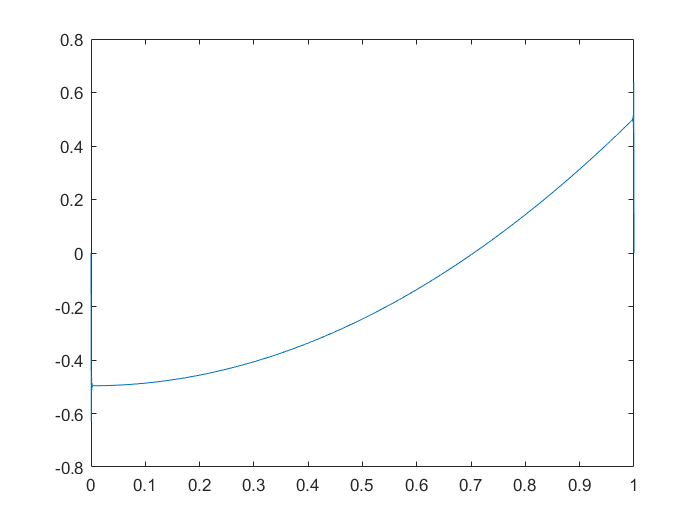}
\end{center}
\caption{ $\varepsilon = 10^{-6}$. Left: Linear, Right: SPLS }
\label{fig:LinearvsSPLS}
\end{figure}

\subsection{Streamline Diffussion  versus PG discretization}


For the second test, we take $f = 2x$ and compare Streamline Diffusion and Petrov-Galerkin. In this case, the exact solution will have a boundary layer at $x = 1$ of width $|\varepsilon\log(\varepsilon)|$. We include two tables for this test where the error is measured only on a subdomain of $[0,1]$ excluding 1\% of the nodes near the right boundary. Table \ref{table:Test2} compares the errors of the SD approximation $u_{h,sd}$ with the PG approximation $u_{h,pg}$ in the SD norm $||u-u_h||_{sd}$. As we can see in Table \ref{table:Test2}, the expected order for SD is observed. Further, the same order is attained by PG, with errors of smaller magnitude. 

\begin{table}[h!] 
\begin{center} 
\begin{tabular}{|*{5}{c|}} 
\hline 
\multicolumn{1}{|c|}{\multirow{2}{*}{\parbox{1.2cm}{\centering Level/$\varepsilon$}}}& \multicolumn{4}{c|}{$10^{-4}$}\\
\cline{2-5} 
&$||u-u_{h,sd}||_{sd}$ &Order&$||u-u_{h,pg}||_{sd}$&Order\\ 
\hline 
1&1.56e-02&-&1.54e-02&-\\ 
\hline 
2&2.57e-03&2.60&2.51e-03&2.62\\ 
\hline 
3&5.45e-04&2.24&4.92e-04&2.35\\ 
\hline 
4&1.05e-04&2.37&4.21e-05&3.55\\ 
\hline 
5&3.86e-05&1.45&1.54e-05&1.46\\ 
\hline 
6&1.45e-05&1.41&5.78e-06&1.41\\ 
\hline 
\multicolumn{1}{|c|}{\multirow{2}{*}{\parbox{1.2cm}{\centering Level/$\varepsilon$}}}& \multicolumn{4}{c|}{$10^{-8}$}\\
\cline{2-5} 
&$||u-u_{h,sd}||_{sd}$ &Order&$||u-u_{h,pg}||_{sd}$&Order\\ 
\hline 
1&1.46e-02&-&1.45e-02&-\\ 
\hline 
2&2.16e-03&2.76&2.09e-03&2.79\\ 
\hline 
3&3.98e-04&2.44&3.16e-04&2.72\\ 
\hline 
4&1.01e-04&1.97&4.04e-05&2.97\\ 
\hline 
5&3.60e-05&1.50&1.43e-05&1.50\\ 
\hline 
6&1.27e-05&1.50&5.06e-06&1.50\\ 
\hline 
\end{tabular} 
\caption{SD vs. PG: $f(x) = 2x$}
\label{table:Test2}
\end{center} 
\end{table} 

In Table \ref{table:Test2} and Table \ref{table:Test3},  the SD and PG approximations are  compared in the optimal norm $||u-u_h||_{*,h}$ for $f(x) = 2x$. The results are interesting as not only does the PG approximation have significantly smaller errors, but also it attains higher order of approximation. In this case ($\varepsilon<h$), in accordance with Theorem \ref{thm:PGError}, $\mathcal{O}(h^2)$ is obtained  for the PG method. 
For the streamline diffusion approximation, the optimal norm only appears to achieve order one, see Table \ref{table:Test3}.

\begin{table}[h!] 
\begin{center} 
\begin{tabular}{|*{5}{c|}} 
\hline 
\multicolumn{1}{|c|}{\multirow{2}{*}{\parbox{1.2cm}{\centering Level/$\varepsilon$}}}& \multicolumn{4}{c|}{$10^{-4}$}\\
\cline{2-5} 
&$||u-u_{h,sd}||_{*,h}$ &Order&$||u-u_{h,pg}||_{*,h}$&Order\\ 
\hline 
1&5.47e-03&0.00&2.32e-03&0.00\\ 
\hline 
2&2.75e-03&0.99&2.68e-04&3.11\\ 
\hline 
3&1.42e-03&0.95&3.71e-05&2.85\\ 
\hline 
4&7.19e-04&0.99&1.61e-06&4.53\\ 
\hline 
5&3.62e-04&0.99&4.27e-07&1.92\\ 
\hline 
6&1.82e-04&0.99&1.21e-07&1.82\\ 
\hline 
\multicolumn{1}{|c|}{\multirow{2}{*}{\parbox{1.2cm}{\centering Level/$\varepsilon$}}}& \multicolumn{4}{c|}{$10^{-8}$}\\
\cline{2-5} 
&$||u-u_{h,sd}||_{*,h}$ &Order&$||u-u_{h,pg}||_{*,h}$&Order\\ 
\hline 
1&5.43e-03&0.00&2.17e-03&0.00\\ 
\hline 
2&2.76e-03&0.98&2.21e-04&3.30\\ 
\hline 
3&1.43e-03&0.95&2.30e-05&3.26\\ 
\hline 
4&7.20e-04&0.99&1.48e-06&3.95\\ 
\hline 
5&3.62e-04&0.99&3.71e-07&2.00\\ 
\hline 
6&1.82e-04&0.99&9.29e-08&2.00\\ 
\hline 
\end{tabular} 
\caption{SD vs. PG:$ f(x) = 2x$}
\label{table:Test3} 
\end{center} 
\end{table}

\begin{table}[h!] 
\begin{center} 
\begin{tabular}{|*{5}{c|}} 
\hline 
\multicolumn{1}{|c|}{\multirow{2}{*}{\parbox{1.2cm}{\centering Level/$\varepsilon$}}}& \multicolumn{4}{c|}{$10^{-4}$}\\
\cline{2-5} 
&$\|u-u_{h,sd}\|_{B}$ &Order&$\|u-u_{h,pg}\|_{B}$&Order\\ 
\hline 
1&1.12e-02&-&2.33e-03&0.00\\ 
\hline 
2&5.74e-03&0.96&2.69e-04&3.12\\ 
\hline 
3&2.92e-03&0.97&3.71e-05&2.85\\ 
\hline 
4&1.47e-03&0.99&1.73e-06&4.43\\ 
\hline 
5&7.38e-04&0.99&4.55e-07&1.92\\ 
\hline 
6&3.70e-04&1.00&1.27e-07&1.84\\ 
\hline 
\multicolumn{1}{|c|}{\multirow{2}{*}{\parbox{1.2cm}{\centering Level/$\varepsilon$}}}& \multicolumn{4}{c|}{$10^{-8}$}\\
\cline{2-5} 
&$\|u-u_{h,sd}\|_{B}$ &Order&$\|u-u_{h,pg}\|_{B}$&Order\\ 
\hline 
1&1.12e-02&-&2.18e-03&0.00\\ 
\hline 
2&5.74e-03&0.96&2.21e-04&3.30\\ 
\hline 
3&2.93e-03&0.97&2.30e-05&3.26\\ 
\hline 
4&1.47e-03&0.99&1.61e-06&3.83\\ 
\hline 
5&7.38e-04&0.99&4.04e-07&2.00\\ 
\hline 
6&3.70e-04&1.00&1.01e-07&2.00\\ 
\hline 
\end{tabular} 
\caption{SD vs. PG:$ f(x) = 2x$}
\label{table:Test4} 
\end{center} 
\end{table}

As reflected by the numerical result presented in Table \ref{table:Test4}, 
the computations using  a mixed balanced norm
 \[
 \|u\|_{B}^2 = (\epsilon+\delta)^2 |u|^2 + \|u\|^2,     
\]
where   $\delta = \frac{2h}{3} $  is parameter from the optimal norm, show   $\mathcal{O}(h)$
order of approximation  for the Streamline Difussion method, and   $\mathcal{O}(h^2)$ for the upwinding  Petrov Galerkin  method. 



\section{Conclusion}\label{sec:conclusion} 
We analyzed and compared four discretization methods for a model\\ convection-diffusion  problem. A unified error analysis was possible because  of our  representation of the optimal norm on the trial spaces. The ideas presented in this paper are  stepping stones for introducing and analyzing new and  efficient discretizations for the multi-dimensional cases of convection dominated problems. 

 Our finite element  analysis  for the considered  model problem showed  that the best method is the upwinding PG method. When compared with  the standard SUPG method, we observed that the upwinding PG method can provide higher order of approximation in particular  norms. In addition, we  proved  that, due to reformulation as a mixed conforming method,  the upwinding PG method leads to  stability and  good approximability  results under less regularity assumptions for the solution.  
 
For the $(P^1-P^2)$-SPLS method, we proved that the seminorm part of the optimal trial norm becomes a norm  which makes the error analysis much simpler. 
 
 Results on  generalizing  the upwinding PG and  SPLS discretizations  for singularly perturbed problems on two or more dimensions are to be discussed in a separate publication. 
 


 
  \bibliographystyle{plain} 
 \def\cprime{$'$} \def\ocirc#1{\ifmmode\setbox0=\hbox{$#1$}\dimen0=\ht0
  \advance\dimen0 by1pt\rlap{\hbox to\wd0{\hss\raise\dimen0
  \hbox{\hskip.2em$\scriptscriptstyle\circ$}\hss}}#1\else {\accent"17 #1}\fi}
  \def\cprime{$'$} \def\ocirc#1{\ifmmode\setbox0=\hbox{$#1$}\dimen0=\ht0
  \advance\dimen0 by1pt\rlap{\hbox to\wd0{\hss\raise\dimen0
  \hbox{\hskip.2em$\scriptscriptstyle\circ$}\hss}}#1\else {\accent"17 #1}\fi}

 \end{document}